\newcommand\R{\mathbb{R}}
\newcommand\bbS{\mathbb{S}}
\newcommand\bbM{\mathbb{M}}

\newcommand\cA{\mathcal{A}}
\newcommand\cT{\mathcal{T}}
\newcommand\cP{\mathcal{P}}

\newcommand{\cF}{\mathcal{F}}
\newcommand{\cJ}{\mathcal{J}}
\newcommand{\cL}{\mathcal{L}}
\newcommand{\cI}{\mathcal{I}}
\newcommand{\cS}{\mathcal{S}}
\newcommand{\cX}{\mathcal{X}}
\newcommand{\cE}{\mathcal{E}}
\newcommand\mt{\mathtt{t}}
\newcommand\bx{\boldsymbol{x}}
\newcommand\bn{\boldsymbol{n}}
\newcommand\bF{\boldsymbol{f}}
\newcommand\bu{\boldsymbol{u}}
\newcommand\bv{\boldsymbol{v}}
\newcommand\bw{\boldsymbol{w}}

\newcommand\beps{\boldsymbol{\varepsilon}}
\newcommand\bsig{\boldsymbol{\sigma}}
\newcommand\btau{\boldsymbol{\tau}}
\newcommand\bze{\boldsymbol{\zeta}}

\newcommand\tD{\mathtt{D}}
\newcommand\tr{\mathop{\mathrm{tr}}\nolimits}

\newcommand\bdiv{\mathop{\mathbf{div}}\nolimits}
\newcommand\sdiv{\mathop{\textup{div}}\nolimits}

\renewcommand\sp{\mathop{\mathrm{sp}}\nolimits}

\newcommand\Ker{\mathop{\mathrm{ker}}\nolimits}

\newcommand\gap{\widehat{\delta}}
\newcommand{\jump}[1]{\llbracket #1 \rrbracket}
\newcommand{\mean}[1]{\{#1\}}
\newcommand{\vertiii}[1]{{\left\vert\kern-0.25ex\left\vert\kern-0.25ex\left\vert #1 
    \right\vert\kern-0.25ex\right\vert\kern-0.25ex\right\vert}}
 
\documentclass[11pt]{article}
\usepackage[utf8]{inputenc}

\usepackage{textcomp}
\usepackage{lmodern}

\usepackage[a4paper,tmargin=3cm,bmargin=3cm, rmargin=1.9cm,lmargin=1.9cm]{geometry}

\usepackage{amssymb, amsmath, mathtools, amsthm, amsfonts}

\usepackage{graphicx, xcolor}
\usepackage{booktabs, multirow}
\usepackage{stmaryrd}
\usepackage{enumitem}

\definecolor{lightblue}{rgb}{0.22,0.45,0.70}
\definecolor{lightgreen}{rgb}{0.22,0.50,0.25}
\usepackage[colorlinks=true,breaklinks=true,linkcolor=lightblue,citecolor=lightgreen,urlcolor=lightblue]{hyperref}
\usepackage[numbers]{natbib}

\DeclarePairedDelimiter\abs{\lvert}{\rvert}
\DeclarePairedDelimiter\norm{\lVert}{\rVert}
\DeclarePairedDelimiter{\inner}{(}{)}
\DeclarePairedDelimiter{\set}{\{}{\}}
\DeclarePairedDelimiter{\dual}{\langle}{\rangle}
\newtheorem{remark}{Remark}[section]
\newtheorem{lemma}{Lemma}[section]
\newtheorem{theorem}{Theorem}[section]
\newtheorem{prop}{Proposition}[section]
\newtheorem{corollary}{Corollary}[section]
\newtheorem{assumption}{Assumption}
\numberwithin{equation}{section}
\numberwithin{figure}{section}
\numberwithin{table}{section}
\date{}
\title{A DG method for a stress formulation of the elasticity eigenproblem with strongly imposed symmetry 
\thanks{This research was  supported by Ministerio de Ciencia e Innovaci\'on, Spain, Project PID2020-116287GB-I00.}}

\author{ 
{\sc Salim Meddahi}\thanks{Facultad de Ciencias, Universidad de Oviedo, Federico Garc\'ia Lorca,  18, 33007-Oviedo, Espa\~na, e-mail: {\tt salim@uniovi.es}.}
}

\begin{document}

\maketitle

\begin{abstract}
\noindent
 We introduce a pure--stress formulation of the elasticity eigenvalue problem with mixed boundary conditions. We propose an H(div)-based discontinuous Galerkin method that imposes strongly the symmetry of the stress  for the discretization of the eigenproblem. Under appropriate assumptions on the mesh and the degree of polynomial approximation,  we demonstrate the spectral correctness of the discrete scheme and derive optimal rates of convergence for eigenvalues and eigenfunctions. Finally,  we provide numerical examples in two and three dimensions.  

\end{abstract}
\medskip

\medskip

\section{Introduction}  

The finite element determination of  the vibration characteristics (natural frequencies and mode shapes) of elastic bodies is of great interest in structural mechanics. For example, the knowledge of the eigenfrequencies keeps the forced oscillations safe from resonance regimes, and the eigenmodes can be used to expand the solution  of transient elastodynamic problems in a Fourier series. We approach this topic from the perspective of the mixed formulation  derived from the Hellinger-Reissner variational principle. Namely, we are interested in variational formulations in which the Cauchy stress tensor prevails as the main unknown. In addition to the fact that accurate approximations of the stress are of paramount importance in many applications, it is well known that mixed formulations are immune to locking in the case of nearly incompressible materials.

In recent years, the theory of Descloux--Nassif--Rappaz \cite{DNR1,DNR2} for non-compact operators has been successfully applied to the mixed finite element analysis of eigenvalue problems in elasticity \cite{MMR,LMMR,meddahi}. The same approach allowed to deal  with   mixed formulations of the Stokes eigenproblem formulated in terms of a pseudo-stress \cite{MMR-stokes, LM} or the Cauchy stress tensor \cite{meddahi}. The symmetry requirement  for the stress tensor, which reflects the conservation of angular momentum, is a specific feature of the Hellinger-Reissner variational principle. The imposition of this restriction in association with H(div)-conformity gives rise to conforming Galerkin methods with a very large number of degrees of freedom, and which are difficult to implement \cite{ArnoldAwanouWinther, hu}. A common practice to overcome this drawback consists in enforcing the symmetry constraint variationally through a Lagrange multiplier. In this context,  \cite{MMR,meddahi} validated the use of the weakly symmetric mixed finite elements \cite{ArnoldFalkWinther,peers,CGG,guzman} for the stress formulation of the elasticity eigenproblem. 

Motivated  by the ability of DG methods to handle efficiently $hp$-adaptive strategies and to facilitate the  implementation of high order methods, an H(div)-based interior penalty version of \cite{MMR} (that retains the weak imposition of the symmetry) has been introduced in  \cite{LMMR}. Nevertheless, on account of \cite{ArnoldAwanouWintherNC, GGNC, Wu}, it is known that relaxing H(div)-conformity by using non-conforming or DG approximations for the elasticity source problem allows the incorporation of the symmetry constraint in the energy space at a reasonable computational cost. To our knowledge, the eigenvalue numerical analysis of such non-conforming/DG mixed methods is not yet available. In this work, our main issue is to determine whether a strong imposition of the symmetry constraint in the scheme introduced in \cite{LMMR} provides a correct eigenvalue approximation. 

The resulting DG-method approximates the stress by symmetric tensors with piecewise polynomial entries of degree $k\geq 1$, in 2D and 3D. We note that, the stress/displacement DG formulation introduced in \cite{Wu} for the elasticity source problem relays on the same discrete space for the stress. However, the displacement field is not present as an independent variable in our DG formulation because it is eliminated via the momentum balance equation. The same equation can be used to post-process the displacement at the discrete level. We prove that the inf-sup stability of the Scott-Vogelius  element \cite{vogelius} for the Stokes problem  (see Assumption~\ref{hyp} below)  is a sufficient condition for the spectral correctness of our DG method. We also obtain optimal error estimates for eigenvalues and eigenfunctions in an adequate DG norm. 

We finally highlight that, unlike \cite{MMR, LMMR}, our analysis does not rely on any extra Sobolev regularity of an auxiliary elasticity source problem. Hence, our analysis remains valid for eigenproblems posed in general domains, with mixed boundary conditions and with minimal requirements on material coefficients.

\medskip 
\noindent\textbf{Outline.}
The contents of this paper have been organized in the following manner. The remainder of this section contains notational conventions and definitions of Sobolev spaces. Section~\ref{sec2} presents the pure--stress formulation of the elasticity eigenproblem and provides a characterization of its spectrum. Preliminary definitions and auxiliary tools related with H(div)-based discontinuous Galerkin methods are collected in Section~\ref{sec3}. The  definition of the mixed DG method (with strong symmetry of the stress)  is detailed in Section~\ref{sec4}, where we also introduce a couple of operators that are useful in our analysis. The spectral correctness of the DG scheme is treated in Section~\ref{sec5}, together with the deduction of optimal error estimates for eigenvalues and eigenspaces. Several numerical results are presented in Section~\ref{sec6}, confirming the expected rates of convergence for different parameter sets including the nearly incompressible regime. 

\medskip 
\noindent\textbf{Notations and Sobolev spaces.} 
We denote the space of real matrices of order $d\times d$ by $\bbM$ and let $\bbS:= \set{\btau\in \bbM;\ \btau = \btau^{\mt} } $  be the subspace of symmetric matrices, where $\btau^{\mt}:=(\tau_{ji})$ stands for the transpose of $\btau = (\tau_{ij})$. The component-wise inner product of two matrices $\bsig, \,\btau \in\bbM$ is defined by $\bsig:\btau:= \sum_{i,j}\sigma_{ij}\tau_{ij}$. We also introduce $\tr\btau:=\sum_{i=1}^d\tau_{ii}$ and denote by $I$ the identity in $\bbM$.
Along this paper we convene to apply all differential operators row-wise.  Hence, given a tensorial function $\bsig:\Omega\to \bbM$ and a vector field $\bu:\Omega\to \R^d$, we set the divergence $\bdiv \bsig:\Omega \to \R^d$, the   gradient $\nabla \bu:\Omega \to \bbM$, and the linearized strain tensor $\beps(\bu) : \Omega \to \bbS$ as
\[
(\bdiv \bsig)_i := \sum_j   \partial_j \sigma_{ij} \,, \quad (\nabla \bu)_{ij} := \partial_j u_i\,,
\quad\hbox{and}\quad \beps(\bu) := \tfrac{1}{2}\left(\nabla\bu+(\nabla\bu)^{\mt}\right).
\]
 
Let $\Omega$ be a polyhedral Lipschitz domain of $\R^d$ $(d=2,3)$, with boundary $\partial \Omega$. For $s\in \R$, $H^s(\Omega,E)$ stands for the usual Hilbertian Sobolev space of functions with domain $\Omega$ and values in E, where $E$ is either $\R$, $\R^d$ or $\bbM$. In the case $E=\R$ we simply write $H^s(\Omega)$. The norm of $H^s(\Omega,E)$ is denoted $\norm{\cdot}_{s,\Omega}$ and the corresponding semi-norm $|\cdot|_{s,\Omega}$, indistinctly for $E=\R,\R^d,\bbM$. We use the convention  $H^0(\Omega, E):=L^2(\Omega,E)$ and let $(\cdot, \cdot)$ be the inner product in $L^2(\Omega, E)$, for $E=\R,\R^d,\bbM$, namely,
\[
  (\bu, \bv):=\int_\Omega \bu\cdot\bv\quad  \forall \bu,\bv\in L^2(\Omega,\R^d),\quad  (\bsig, \btau):=\int_\Omega \bsig:\btau \quad \forall \bsig, \btau\in L^2(\Omega,\bbM). 
\]
 

We consider the space $H(\bdiv, \Omega, E)$ of tensors $\btau \in L^2(\Omega, E)$ satisfying $\bdiv \btau \in L^2(\Omega, \R^d)$, and denote the corresponding norm  $\norm{\btau}^2_{H(\bdiv,\Omega)}:=\norm{\btau}_{0,\Omega}^2+\norm{\bdiv\btau}^2_{0,\Omega}$, where $E$ is either $\bbM$ or $\bbS$. Let $\bn$ be the outward unit normal vector to {$\partial \Omega$}. Let $\btau$ be a sufficiently regular symmetric tensor, Green's formula
\begin{equation}\label{Gf}
	(\btau, \beps(\bv)) + (\bdiv \btau, \bv) = \int_{\partial \Omega} \btau\bn\cdot \bv\qquad  \bv \in H^1(\Omega,\R^d),
\end{equation}
can be used to extend the normal trace operator $\btau \to (\btau|_{\partial \Omega})\bn$ to a linear continuous mapping $(\cdot|_{\partial \Omega})\bn:\, H(\bdiv, \Omega, \bbS) \to H^{-\frac{1}{2}}(\partial \Omega, \R^d)$, where $H^{-\frac{1}{2}}(\partial \Omega, \R^d)$ is the dual of $H^{\frac{1}{2}}(\partial \Omega, \R^d)$.

Throughout this paper, we shall use the letter $C$ to denote a generic positive constant independent of the mesh size  $h$, that may stand for different values at its different occurrences. Moreover, given any positive expressions $X$ and $Y$ depending on $h$, the notation $X \,\lesssim\, Y$  means that $X \,\le\, C\, Y$. 

\section{A stress formulation of the elasticity eigenproblem}\label{sec2}
Our aim is to determine the natural frequencies $\omega\in \R$ of an elastic structure with  mass density $\varrho$ and occupying a polyhedral Lipschitz domain  $\Omega\subset \mathbb R^d$ ($d=2,3$) . This amounts to solve the eigenproblem 
\begin{align}
 \bdiv \bsig + \omega^2\varrho(\bx) \bu &= 0 \quad \text{in $\Omega$}, \label{modela}
 \\[1ex]
 \cA(\bx)\bsig &= \beps(\bu) \quad \text{in $\Omega$}, \label{modelb}
\end{align}
where  $\bu:\Omega\times[0, T] \to \R^d$ is the displacement field and $\bsig:\Omega\to \bbS$ is the Cauchy stress tensor. The symmetric and positive-definite 4$^{th}$-order tensor $\cA(\bx): \bbS\to \bbS$ involved in the linear material law \eqref{modelb} is known as the compliance tensor. We assume that there exist $a^+ > a^->0$ such that  
\[
  a^- \bze:\bze\, \leq \cA(\bx) \bze :\bze \leq a^+ \, \bze:\bze\quad \forall \bze\in \bbS,\quad \text{a.e. in $\Omega$}.  
\]
We also suppose that there exists a polygonal/polyhedral disjoint partition $\big\{\bar\Omega_j,\ j= 1,\ldots,J\big\}$ of  $\bar \Omega$  such that $\varrho|_{\Omega_j}:= \varrho_j>0$ for all $j=1,\ldots,J$ and let $\varrho^+:= \max_j \varrho_j$ and $\varrho^-:= \min_j \varrho_j$. 

We impose the boundary condition $\bu=\mathbf 0$ on a subset $\Gamma_D\subset\Gamma:= \partial \Omega$ of positive surface measure and let the structure free of stress on $\Gamma_N:= \Gamma \setminus \Gamma_D$. Here, we opt for combining the equilibrium equation \eqref{modela} with the constitutive law \eqref{modelb} to eliminate the displacement field $\bu$ and impose $\bsig$ as a primary variable. This procedure leads to the following  eigensystem:  find eigenmodes $0 \neq \bsig:\Omega\to \bbS$ and eigenfrequencies  $\omega\in \mathbb{R}$ such that,
\begin{align}\label{modelPb}
\begin{split}
  -\beps\left(\tfrac{1}{\varrho} \bdiv  \bsig \right) &= \omega^2  \cA\bsig     \quad \text{ in $\Omega$},
 \\
\tfrac{1}{\varrho} \bdiv \bsig &= 0  \quad \text{ on $\Gamma_D$},
\\
\bsig\bn &= \mathbf 0 \quad \text{ on $\Gamma_N$},
\end{split}
\end{align}
where $\bn$ stands for the exterior unit normal vector on $\Gamma$.

In the following, we write $H$ for the space $L^2(\Omega,\bbS)$ endowed with the $\cA$-weighted inner product $\inner*{\bsig,\btau}_{\cA} := \inner*{\cA\bsig,\btau}$ and denote the corresponding norm $\norm*{\btau}^2_\cA:= \inner*{\cA\btau,\btau}$.  The eigenfunctions $\bsig$ will be sought in the closed subspace $X$ of $H(\bdiv, \Omega, \bbS)$ defined by
\[
X := \set*{\btau\in H(\bdiv, \Omega, \bbS); \quad 
	\dual*{\btau\bn,\boldsymbol{\phi}}_{\Gamma}= 0 
	\quad \text{$\forall\boldsymbol{\phi} \in H^{1/2}(\Gamma,\R^d)$,\, $\boldsymbol{\phi}|_{\Gamma_D} = \mathbf{0}$}},
\]
where $\dual*{\cdot, \cdot}_\Gamma$ holds for the duality pairing between $H^{\frac{1}{2}}(\Gamma,\R^d)$ and $H^{-\frac{1}{2}}(\Gamma,\R^d)$. We introduce the symmetric and positive semidefinite bilinear form $c:\, X\times X \to \R$ given by 
\[
   c(\bsig, \btau) := \inner*{\tfrac{1}{\varrho}\bdiv\bsig, \bdiv\btau}
\]
and endow $X$ with the Hilbertian inner product $a\inner*{\bsig, \btau} := \inner*{\bsig,\btau}_{\cA} + c(\bsig, \btau)$. We denote the corresponding norm $\norm{\btau}^2_X := a(\btau, \btau)$.

Testing the first equation of \eqref{modelPb} with $\btau\in X$ and applying Green's formula \eqref{Gf} we deduce, after a shift argument, the following  pure--stress variational formulation of the eigenproblem: find $0\ne \bsig \in X$ and $\kappa= 1 + \omega^2 \in \R$ such that
\begin{equation}\label{S}
	a\inner*{\bsig,\btau} = \kappa \inner*{\bsig,\btau}_{\cA}, \quad \forall \btau \in X. 
\end{equation}

We introduce the source operator $\tilde T:\, L^2(\Omega, \bbS)\to X$ corresponding to the variational eigenproblem \eqref{S}; which is  defined for any $\bF\in L^2(\Omega, \bbS)$ by 
	\begin{equation}\label{source}
		a\inner*{\tilde T\bF,\btau} =  \inner*{\bF,\btau}_{\cA}, \quad \forall \btau \in X. 
	\end{equation}
Obviously, $\tilde T$ is linear and bounded, actually it holds,  
\begin{equation}\label{cotaT}
	\norm*{\tilde T\bF}_X \leq \norm*{\bF}_\cA\quad \forall \bF\in H. 
\end{equation}
We denote the $H^1$-Sobolev space with incorporated Dirichlet boundary conditions on either $\Gamma_D$ or $\Gamma_N$ by 
\[
  H^1_\star(\Omega,\R^d):= \set*{\bv \in H^1(\Omega,\R^d);\ v|_{\Gamma_\star} = \mathbf 0},\quad \star \in \set*{D, N}.
\]
It is important to notice that testing \eqref{source} with a tensor $\btau:\Omega \to \bbS$ whose entries are  indefinitely differentiable and compactly supported  in $\Omega$ proves that  $\beps(\frac{1}{\varrho} \bdiv (\tilde T\bF)) = \cA(\tilde T - I)\bF \in L^2(\Omega,\bbS)$. Hence, by virtue of Korn's inequality, $\frac{1}{\varrho} \bdiv (\tilde T\bF)\in H^1(\Omega,\R^d)$ and it follows readily from Green's formula \eqref{Gf} that $\frac{1}{\varrho} \bdiv (\tilde T\bF)$ vanishes on $\Gamma_D$. In other words, $\frac{1}{\varrho} \bdiv (\tilde T\bF) \in H^1_D(\Omega,\R^d)$ and  there exists $C>0$ such that 
	\begin{equation}\label{regT}
	\norm*{\tfrac{1}{\varrho} \bdiv (\tilde T\bF)}_{1,\Omega} \leq C \norm{\bF}_\cA, \quad \forall \bF\in H. 	
	\end{equation}
 The operator $T:=\tilde T|_X:\, X \to X$ is relevant in our analysis because its eigenvalues and those of  problem \eqref{S} are reciprocal to each other and the corresponding eigenfunctions are the same. A full description of the spectrum of $T$ will then solve problem \eqref{S}. 
   
We consider the direct sum decomposition $X = K\oplus K^\bot$  into closed subspaces
\[
K := \set*{\btau\in X;\ \bdiv \btau = 0 \ \text{in $\Omega$}} \quad \text{and} \quad K^\bot:= \set*{\bsig \in X;\ \inner*{\bsig, \btau}_{\cA} = 0, \ \forall \btau \in K},  
\] 
 which are orthogonal with respect to both $(\cdot, \cdot)_{\cA}$ and $a(\cdot, \cdot)$. It is clear that $\kappa=1$ is an eigenvalue of \eqref{S} with associated eigenspace $K$. Consequently, as $K$ is not a finite-dimensional  subspace of $X$, $T$ is not a compact operator.

\begin{lemma}\label{P}
The orthogonal projection $P$ in $X$ onto $K^\bot$ is characterized, for any $\bsig\in X$, by  $P\bsig :=\widetilde \bsig$ where $\widetilde \bsig = \cA^{-1} \beps(\widetilde \bu)$ and  $\widetilde\bu\in H^1_{D}(\Omega,\R^d)$ is the unique solution of 
\begin{equation}\label{korn}
	\inner*{\cA^{-1}\beps(\widetilde{\bu}), \beps(\bv)} = -\inner*{\bdiv\bsig, \bv},\quad \forall \bv \in H^1_D(\Omega, \R^d).
\end{equation}
\end{lemma}
\begin{proof}
	We first point out that Korn's inequality provides the stability estimate  
\begin{equation}\label{stabi} 
	\norm*{\tilde \bu}_{1,\Omega}  \leq C \norm*{\bdiv \bsig}_{0,\Omega}.
\end{equation}
By definition, \eqref{stabi} also ensures that $\norm*{P\bsig}_{0,\Omega} \leq C_1 \norm*{\bdiv \bsig}_{0,\Omega}$. Moreover, $\bdiv P \bsig = \bdiv \bsig$ by construction, which ensures that $P:\, X\to X$ is bounded. Moreover,  it is clear that $P\circ P = P$ and $\Ker{P} = K$. It remains to show that the range of $P$ coincides with $K^\bot$. To this end, we notice that, for any $\bsig\in X$,
\[
  \inner*{P\bsig, \btau}_{\cA} = \inner*{\beps(\widetilde \bu) , \btau }= (\nabla \widetilde \bu, \btau) = 0,\quad  \forall \btau \in K,
\]
which proves that $P(X)\subset K^{\bot}$. The reciprocal inclusion is a consequence of $K^{\bot} = P (K^{\bot}) + (I - P)K^{\bot} = P (K^{\bot}) \subset P(X)$, where we used that  $(I - P)X \subset K$, and the result follows.  

\end{proof}

\begin{lemma}\label{compact}
	The inclusions $P(X)\hookrightarrow H$ and $P(X)\cap T(X)\hookrightarrow X$ are compact.
\end{lemma}
\begin{proof}
	Let $\set*{\bsig_n}_n$ be a weakly convergent sequence  in $X$. The continuiuty of $P:X\to X$ implies that the sequence $\set*{\widetilde\bsig_n}_n:=\set*{P\bsig_n}_n$  is also weakly convergent in $X$. By definition, $\widetilde\bsig_n = \cA^{-1} \beps(\widetilde \bu_n)$, where $\widetilde \bu_n\in H^1_{D}(\Omega, \R^d)$ solves \eqref{korn} with right-hand side $-\bdiv\bsig_n$. It follows from \eqref{stabi} that $\widetilde \bu_n$ is bounded in $H^1_{D}(\Omega, \R^d)$ and the compactness of the embedding $H^1(\Omega,\R^d) \hookrightarrow  L^2(\Omega, \R^d)$ implies that $\set*{\widetilde \bu_n}_n$ admits a subsequence (denoted again $\set*{\widetilde \bu_n}_n$ ) that converges strongly in $L^2(\Omega,\R^d)$. Next, we deduce from Green's identity  
	\[
  \inner*{\widetilde \bsig_p - \widetilde \bsig_q, \widetilde \bsig_p - \widetilde \bsig_q}_{\cA} = \inner*{\beps(\widetilde \bu_p - \widetilde \bu_q), \widetilde \bsig_p - \widetilde \bsig_q}  = - \inner*{\widetilde \bu_p - \widetilde \bu_q, \bdiv(\widetilde \bsig_p - \widetilde \bsig_q)},
\]
that $\set*{\widetilde\bsig_n}_n$ is a Cauchy sequence in $H$, which implies that the  embedding $P(X)\hookrightarrow H$ is compact. 

Finally, it follows from \eqref{regT} that  
\[
  T(X)\cap P(X) \subset \set*{\bsig \in P(X);\ \tfrac{1}{\varrho} \bdiv \bsig \in H^1(\Omega, \R^d)},
\]
and the compactness of the embedding $T(X)\cap P(X)\hookrightarrow X$ is a consequence of the fact that the inclusion $\set*{\bsig \in P(X);\ \tfrac{1}{\varrho} \bdiv \bsig \in H^1(\Omega, \R^d)} \subset X$ is compact.
\end{proof}


We point out that $\tilde T$ is symmetric with respect to $(\cdot, \cdot)_{\cA}$, which implies that $P(X) = K^{\bot}$ is $T$-invariant. Consequently, it holds $T(P(X)) \subset P(X)\cap T(X)$ and Lemma~\ref{compact} implies that the $a(\cdot, \cdot)$-symmetric and positive definite operator $T:\, K^{\bot} \to K^{\bot}$ is compact. Therefore, we have the following characterization  of the spectrum of $T$.

\begin{theorem}\label{specT}
The spectrum $\sp(T)$ of $T$ is given by  $\sp(T) = \set{0, 1} \cup \set{\eta_k}_{k\in \mathbb{N}}$, where $\set{\eta_k}_k\subset (0,1)$ is a sequence of finite-multiplicity eigenvalues of $T$ that converges to 0.  The ascent of each of these eigenvalues is $1$ and the corresponding eigenfunctions lie in $P(X)$. Moreover, $\eta=1$ is an infinite-multiplicity eigenvalue of $T$ with associated eigenspace $K$ and  $\eta=0$ is not an eigenvalue.
\end{theorem}

\section{Definitions and auxiliary results}\label{sec3}

We consider  a sequence $\{\mathcal{T}_h\}_h$ of shape-regular simplicial meshes that subdivide the domain $\bar \Omega$ into  simplices  $K$ of diameter $h_K$. The parameter $h:= \max_{K\in \cT_h} \{h_K\}$ represents the mesh size of $\cT_h$.  We assume that $\mathcal{T}_h$ is aligned with the partition $\bar\Omega = \cup_{j= 1}^J \bar{\Omega}_j$ and that $\cT_h(\Omega_j) := \set*{K\in \cT_h;\ K\subset \Omega_j }$ is a shape-regular mesh of $\bar\Omega_j$ for all $j=1,\cdots, J$ and  all $h$.

For all $s\geq 0$, we consider the broken Sobolev space    
\[
  H^s(\cup_j\Omega_j) := \set*{ v\in L^2(\Omega);\ v|_{\Omega_j}\in H^s(\Omega_j),\ \forall j =1,\ldots,J }
\]
corresponding to the partition $\bar\Omega = \cup_{j= 1}^J \bar{\Omega}_j$. 
Its vectorial and tensorial versions are denoted $H^s(\cup_j\Omega_j,\R^d)$ and $H^s(\cup_j\Omega_j,\bbM)$, respectively. Likewise, the broken Sobolev space with respect to the subdivision of $\bar \Omega$ into $\cT_h$ is  
\[
 H^s(\cT_h,E):=
 \set*{\bv \in L^2(\Omega, E): \quad \bv|_K\in H^s(K, E)\quad \forall K\in \cT_h},\quad \text{for $E \in \set{ \R, \R^d, \bbM}$}. 
\]
For each $\bv:=\set{\bv_K}\in H^s(\cT_h,\R^d)$ and $\btau:= \set{\btau_K}\in H^s(\cT_h,\bbM)$ the components $\bv_K$ and $\btau_K$  represent the restrictions $\bv|_K$ and $\btau|_K$. When no confusion arises, the restrictions of these functions will be written without any subscript.

Hereafter, given an integer $m\geq 0$ and a domain $D\subset \mathbb{R}^d$, $\cP_m(D)$ denotes the space of polynomials of degree at most $m$ on $D$. We introduce the space   
\[
 \cP_m(\cT_h) := 
 \set{ v\in L^2(\Omega): \ v|_K \in \cP_m(K),\ \forall K\in \cT_h }
 \]
 of piecewise polynomial functions relatively to $\cT_h$. We also consider the space $\cP_m(\cT_h,E)$ of functions with values in $E$ and entries in $\cP_m(\cT_h)$, where $E$ is either $\R^d$, $\bbM$ or $\bbS$. 
 
 Let us introduce now notations related to DG approximations of $H(\text{div})$-type spaces. We say that a closed subset $F\subset \overline{\Omega}$ is an interior edge/face if $F$ has a positive $(d-1)$-dimensional measure and if there are distinct elements $K$ and $K'$ such that $F =\bar K\cap \bar K'$. A closed subset $F\subset \overline{\Omega}$ is a boundary edge/face if there exists $K\in \cT_h$ such that $F$ is an edge/face of $K$ and $F =  \bar K\cap \Gamma$. We consider the set $\cF_h^0$ of interior edges/faces, the set $\cF_h^\partial$ of boundary edges/faces and let $\cF(K):= \set{F\in \cF_h;\quad F\subset \partial K}$ be the set of edges/faces composing the boundary of $K$. We assume that the boundary mesh $\cF_h^\partial$ is compatible with the partition $\partial \Omega = \Gamma_D \cup \Gamma_N$ in the sense that, if 
$
\cF_h^D = \set*{F\in \cF_h^\partial:\, F\subset \Gamma_D}$  and $\cF_h^N = \set*{F\in \cF_h^\partial:\, F\subset \Gamma_N},
$
then $\Gamma_D = \cup_{F\in \cF_h^D} F$ and $\Gamma_N = \cup_{F\in \cF_h^N} F$. We denote   
\[
  \cF_h := \cF_h^0\cup \cF_h^\partial\qquad \text{and} \qquad \cF^*_h:= \cF_h^{0} \cup \cF_h^{N},
\]
and for all $K\in \cT_h$. Obviously, in the case $\Gamma_D = \Gamma$ we have that $\cF^*_h = \cF^0_h$. 

 We will need the space given on the skeletons of the triangulations $\cT_h$  by $L^2(\cF^*_h):= \bigoplus_{F\in \cF^*_h} L^2(F)$. Its vector valued version is denoted $L^2(\cF^*_h,\R^d)$. Here again, the components $\bv_F$ of $\bv := \set{\bv_F}\in L^2(\cF^*_h,\R^d)$  coincide with the restrictions $\bv|_F$.  We endow $L^2(\cF^*_h,\R^d)$ with the inner product 
\[
(\bu, \bv)_{\cF^*_h} := \sum_{F\in \cF^*_h} \int_F \bu_F\cdot \bv_F\quad \forall \bu,\bv\in L^2(\cF^*_h,\R^d),
\]
and denote the corresponding norm $\norm*{\bv}^2_{0,\cF^*_h}:= (\bv,\bv)_{\cF^*_h}$. From now on, $h_\cF\in L^2(\cF^*_h)$ is the piecewise constant function defined by $h_\cF|_F := h_F$ for all $F \in \cF^*_h$ with $h_F$ denoting the diameter of edge/face $F$. By virtue of our hypotheses on $\varrho$ and on the triangulation $\cT_h$, we may consider that $\varrho$ is an element of $\cP_0(\cT_h)$ and denote $\varrho_K:= \varrho|_K$ for all $K\in \cT_h$. We introduce $\varrho_\cF\in L^2(\cF^*_h)$ defined by $\varrho_F := \min\{ \varrho_K, \varrho_{K'} \}$ if $K\cap K' = F$ and $\varrho_F := \varrho_K$ if $F\cap K \in \cF^N_h$. 

Given  $\bv\in H^s(\cT_h,\R^d)$ and $\btau\in H^s(\cT_h,\bbM)$, with $s>\frac{1}{2}$, we define averages $\mean{\bv}\in L^2(\cF^*_h,\R^d)$ and jumps $\jump{\btau}\in L^2(\cF^*_h,\R^d)$ by
\[
 \mean{\bv}_F := (\bv_K + \bv_{K'})/2 \quad \text{and} \quad \jump{\btau}_F := 
 \btau_K \bn_K + \btau_{K'}\bn_{K'} 
 \quad \forall F \in \cF(K)\cap \cF(K'),
\]
with the conventions 
\[
 \mean{\bv}_F := \bv_K  \quad \text{and} \quad \jump{\btau}_F := 
 \btau_K \bn_K  
 \quad \forall F \in \cF(K),\,\, F\in \cF_h^\partial,
\]
where $\bn_K$ is the outward unit normal vector to $\partial K$.

For any $k\geq 1$, we 
let $X^{DG}(h) :=X + X_h^{DG}$, with $X_h^{DG}:=\cP_{k}(\cT_h,\bbS)$.  Given $\btau \in X_h^{DG}$, we define $\bdiv_h \btau \in  L^2(\Omega,\R^d)$ by $\bdiv_h \btau|_{K} := \bdiv \btau_K$ for all $K\in \cT_h$ and endow $X_k^{DG}(h)$ with the norm
\[
 \vertiii{\btau}^2 := \norm*{\btau}^2_\cA + \norm*{\tfrac{1}{\sqrt{\varrho}}\bdiv_h  \btau}^2_{0,\Omega} + \norm*{\varrho_\cF^{-\frac{1}{2}} h_{\cF}^{-\frac{1}{2}} \jump{\btau}}^2_{0,\cF^*_h}.
\]
If it happens that $\bdiv_h \btau \in H^s(\cT_h,\R^d)$ with $s>\frac{1}{2}$,  we also introduce
\[
  \vertiii{\btau}^2_* := \vertiii{\btau}^2 + \norm*{\varrho_\cF^{\frac{1}{2}} h_F^{\frac{1}{2}} \mean{\tfrac{1}{\varrho}\bdiv_h \btau}}^2_{0,\cF^*_h}.
\]
It is important to notice that $\vertiii{\btau} = \norm*{\btau}_X$ for all $\btau\in X$.  

The following discrete trace inequality is useful in our analysis. 
\begin{lemma}\label{TraceDG}
There exists a constant $C_{\textup{tr}}>0$ independent of $h$ and $\varrho$ such that 	\begin{equation}\label{discTrace}
  \norm*{\varrho^{\frac{1}{2}}_\cF h^{\frac{1}{2}}_{\cF}\mean{\tfrac{1}{\sqrt \varrho} \bv}}_{0,\cF^*_h}\leq C_{\textup{tr}} \norm*{\tfrac{1}{\sqrt \varrho} \bv}_{0,\Omega}\quad \forall  \bv\in \cP_k(\cT_h, \R^d). 
 \end{equation}
\end{lemma}
\begin{proof}
By definition of $\varrho_\cF$, for any $\bv \in \cP_k(\cT_h, \R^d)$, it holds
	\begin{align*}
		\norm*{\varrho^{\frac{1}{2}}_\cF h^{\frac{1}{2}}_{\cF}\mean{\tfrac{1}{\varrho} \bv}}^2_{0,\cF^*_h} = \sum_{F\in \cF^*_h} h_F \norm*{ \varrho^\frac{1}{2}_{F}\mean{\tfrac{1}{\varrho} \bv}_F }^2_{0,F} \leq \sum_{F\in \cF^*_h} h_F \norm*{ \mean{\tfrac{1}{\sqrt \varrho} \bv}_F }^2_{0,F}
		\lesssim \sum_{K\in \cT_h} h_K \norm*{ \tfrac{1}{\sqrt \varrho} \bv_K }^2_{0,\partial K}.
	\end{align*}
Applying in the last inequality the well-known estimate  (see for example \cite{DiPietroErn})
	\begin{equation}\label{Tr0}
		h_K^{\frac{1}{2}}\norm{\phi}_{0,\partial K} \leq C_{\textup{tr}}  \norm{\phi}_{0,K}\quad  \forall \phi \in \cP_k(K),
	\end{equation}
where $C_{\textup{tr}}>0$ is independent of $h$, we obtain the result.
\end{proof}

For all $\bsig,\btau \in X^{DG}(h)$ and for a large enough given parameter $\mathtt{a}>0$, we consider the symmetric bilinear form
\[
  c_h\inner*{\bsig, \btau} := c(\bsig, \btau) + \mathtt{a} \inner*{ \varrho_\cF^{-1}h_\cF^{-1}\jump{\bsig}, \jump{\btau}}_{\cF^*_h}
  - \inner*{\mean{\tfrac{1}{\varrho} \bdiv_h \bsig}, \jump{\btau} }_{\cF^*_h}
	- \inner*{\mean{\tfrac{1}{\varrho} \bdiv_h \btau}, \jump{\bsig} }_{\cF^*_h}
\]
and let
\[
  a_h\inner*{\bsig, \btau} := \inner*{\cA \bsig,\btau} +  c_h\inner*{\bsig, \btau}.
\]
For all $\bsig, \btau \in X^{DG}(h)$
satisfying  $\bdiv_h\bsig,\bdiv_h\btau\in H^s(\cT_h,\R^d)$ with $s>1/2$,
a straightforward application of the Cauchy-Schwarz inequality gives 
\begin{equation*}
\left|a_h(\bsig,\btau)\right| \leq 2 \vertiii{\bsig}_*\, \vertiii{\btau}_*.
\end{equation*}
Moreover, if we take in the last estimate $\btau= \btau_h\in X_h^{DG}$, we deduce from Lemma~\ref{TraceDG} that, 
\begin{equation}\label{boundA2}
\left|a_h(\bsig, \btau_h)\right| \leq M \,  \vertiii{\bsig}_*\, \vertiii{\btau_h},
\end{equation}
with $M := 2 \sqrt{1 + C_{\textup{tr}}^2}$.

The bilinear form $c_h(\cdot, \cdot)$ and the DG-norm $\vertiii{\cdot}$ are designed in such a way that the coercivity of the bilinear form $a_h(\cdot, \cdot)$ on $X_h^{DG}$ can be achieved with a stability parameter $\mathtt{a}$ that is independent of the material coefficients, as shown in the following result. 

\begin{prop}\label{SPD}
There exists a constant $\mathtt{a}^*>0$, independent of $\varrho$ and $\cA$, such that  
if $\mathtt{a} \geq  \mathtt{a}^*$, then
\begin{equation}\label{ste}
  c_h(\btau, \btau) \geq \frac{1}{2} \left(  \norm*{\varrho^{-\frac{1}{2}} \bdiv_h \btau}^2_{0,\Omega} +  \norm*{ \delta_\cF^{-\frac{1}{2}} h_\cF^{-\frac{1}{2}}\jump{\btau}}^2_{0,\cF^*_h}\right) ,\quad \forall \btau \in X_h^{DG}.
\end{equation}
\end{prop}
\begin{proof}
By definition, we have 
\begin{equation}\label{ste00}
	 c_h(\btau, \btau) = \norm*{\varrho^{-\frac{1}{2}} \bdiv_h \btau}^2_{0,\Omega} + \mathtt{a} \norm*{ \varrho_\cF^{-\frac{1}{2}} h_\cF^{-\frac{1}{2}}\jump{\btau}}^2_{0,\cF^*_h}
	 - 2 \inner*{\mean{\varrho^{-1} \bdiv_h \btau}, \jump{\btau} }_{\cF^*_h}
\end{equation}
Using the Cauchy-Schwarz inequality, Young's inequality together with the discrete trace inequality \eqref{discTrace} we obtain the estimate  
\begin{align}\label{es0}
\begin{split}
	2 \left| \inner*{ \mean{\varrho^{-1} \bdiv_h \btau}, \jump{\btau} }_{\cF_h^0} \right| &\leq  2\norm*{\gamma_\cF^{\frac{1}{2}}h_\cF^{\frac{1}{2}} \mean{\varrho^{-1} \bdiv_h \btau} }_{0,\cF^*_h} \norm*{\gamma_\cF^{-\frac{1}{2}} h_\cF^{-\frac{1}{2}} \jump{\btau} }_{0,\cF^*_h}
	\\[1ex]
&\leq 2 C_{\textup{tr}}\norm*{\varrho^{-\frac{1}{2}} \bdiv \btau}_{0,\Omega} \norm*{\gamma_\cF^{-\frac{1}{2}} h_\cF^{-\frac{1}{2}} \jump{\btau}}_{0,\cF^*_h}
\\[1ex]
&\leq \frac{1}{2} \norm*{\varrho^{-\frac{1}{2}} \bdiv \btau}^2_{0,\Omega} + 2 C_{\textup{tr}}^2 \norm*{\gamma_\cF^{-\frac{1}{2}} h_\cF^{-\frac{1}{2}} \jump{\btau}}^2_{0,\cF^*_h}. 
\end{split}
\end{align} 
Combining \eqref{es0} and \eqref{ste00} gives the result with $\mathtt{a}^*:=2 C_{\textup{tr}}^2 + \frac{1}{2}$. 
\end{proof}

\section{The pure--stress DG scheme}\label{sec4}

We are now in a position to introduce the following mixed DG discretization of \eqref{S}: Find  $0\ne \bsig_h \in X_h^{DG}$ and $\kappa_h \in \R$ such that
\begin{equation}\label{Sh}
	a_h\inner*{\bsig_h, \btau} = \kappa_h \inner*{\bsig_h,\btau}_\cA, \quad \forall \btau \in X_h^{DG}. 
\end{equation}

\begin{remark}
We are only considering via \eqref{Sh} the symmetric interior penalty DG method (SIP) because the non-symmetric DG versions, known in the literatura as NIP and IIP, have sup-optimal rates of convergence for the eigenvalues \cite{Antonietti,LM}. 	
\end{remark}

 In all what follows, we make the following stability assumption.  
 \begin{assumption}\label{a0}
The parameter $\mathtt{a}$ is greater than or equal to $\mathtt{a}^*$:  $\mathtt{a} \geq  \mathtt{a}^*:=2 C_{\textup{tr}}^2 + \frac{1}{2}$.	  
\end{assumption}
Under this assumption,  Proposition permits us to guaranty the well--posedness of the discrete source operator $\tilde T_h:H \to X_h^{DG}$ given, for any $\bF\in H$, by
\begin{equation}\label{charcTDG}
 a_h(\tilde T_h \bF, \btau_h ) = \inner*{\bF, \btau_h}_\cA \quad \forall \btau_h\in X_h^{DG}.
\end{equation}
Actually, $\tilde T_h$ is uniformly bounded, namely,   
\begin{equation}\label{bTh}
  \vertiii{\tilde T_h \bF} \leq 2 \norm{\bF}_\cA,\quad \forall \bF\in H.
\end{equation}
 Similarly to the continuous case, we observe that $(1 \neq\kappa_h,\bsig_h)\in\mathbb{R}\times X_h^{DG}$ is a solution of problem \eqref{Sh} if and only if $(\frac{1}{\kappa_h},\bsig_h)$,
 is an eigenpair of $T_h:=\tilde T_h|_{X_h^{DG}}$, i.e., $T_h\bsig_h=\frac{1}{\kappa_h}\bsig_h$. Moreover, it is clear that $\kappa_h=1$ is an eigenvalue common to \eqref{Sh} and $T_h$ with corresponding eigenspace 
 \begin{equation}\label{kerTh}
  K_h := \set*{\btau_h\in X_h^{DG};\ c_h(\btau_h, \btau_h) = 0 }. 
\end{equation}

The following result establishes a C\'ea estimate for the DG approximation \eqref{charcTDG} of \eqref{source}.
\begin{theorem}\label{cea}
Under Assumption~\ref{a0}, for all $\bF\in H$, it holds 
\begin{equation}\label{Cea}
\vertiii{ (\tilde T - \tilde T_h) \bF } \leq ( 1 + 2 M ) 
\inf_{\btau_h\in X_h^{DG}} \vertiii{\tilde T \bF - \btau_h}_*,
\end{equation}
with $M$ as in \eqref{boundA2}.
\end{theorem}
\begin{proof}
We already know from \eqref{regT} that $\bu := \frac{1}{\varrho} \bdiv (\tilde T\bF)\in H_D^1(\Omega,\R^d)$. Hence, using the integration by parts \eqref{Gf} elementwise  in
	\[
  a(\tilde T \bF, \btau) = \inner*{\bF, \btau}_\cA,\quad \forall \btau \in X
\]
 gives  
\begin{align*}
	\inner*{\tilde T\bF - \bF, \btau}_\cA =   \inner*{\beps(\bu), \btau} 
    =  -\inner*{\bu, \bdiv_h\btau } + \inner*{\mean{\bu}, \jump{\btau} }_{\cF^*_h} \quad   \forall \btau\in X_h^{DG}.
\end{align*}
Substituting back $\bu = \frac{1}{\varrho} \bdiv (\tilde T\bF)$ into the last expression we get
\[
  \inner*{\tilde T\bF , \btau}_\cA + \inner*{\tfrac{1}{\varrho} \bdiv (\tilde T\bF), \bdiv_h\btau }
  - \inner*{\mean{\tfrac{1}{\varrho} \bdiv (\tilde T\bF)}, \jump{\btau} }_{\cF^*_h} = \inner*{\bF, \btau}_\cA\quad   \forall \btau\in X_h^{DG}.
\]
Combining the last identity with \eqref{charcTDG} yields the following consistency property \begin{equation}\label{consistency}
a_h((\tilde T - \tilde T_h)\bF, \btau_h) = 0 \quad \forall \btau_h\in X_h^{DG},\quad \forall \bF\in H.
\end{equation}
Now, by virtue of \eqref{boundA2}, \eqref{ste} and \eqref{consistency}, it holds  
\begin{align*}
	\tfrac{1}{2} \vertiii{\tilde T_h \bF - \btau_h}^2 &\leq a_h(\tilde T_h \bF - \btau_h, \tilde T_h \bF - \btau_h) = a_h(\tilde T \bF - \btau_h, \tilde T_h \bF - \btau_h)
	\\[1ex]
	& \leq M\, \vertiii{\tilde T \bF - \btau_h}_* \vertiii{\tilde T_h \bF - \btau_h} \quad \forall \btau_h\in X_h^{DG},
\end{align*}
and the result follows from the triangle inequality. 

\end{proof}

\subsection{The operator $\cJ_h^s$}

For technical reasons, we want to consider here the $H(\bdiv,\Omega,\bbS)$-conforming finite element space given by $X_h^c:= \cP_h(\cT_h, \bbS)\cap X$. The goal of this section is to prove that, under certain conditions on the mesh and on the polynomial degree $k$, it holds
\begin{equation}\label{approx}
	   \inf_{\btau_h\in X_h^c} \norm*{\bsig - \btau_h}_X \longrightarrow 0, \quad \text{when $h\to 0$},\quad \forall \bsig \in X.
\end{equation}

The main obstacle in performing this task is the symmetry constraint. Let us ignore this constraint and discuss, in a first step, approximation properties of the Brezzi-Douglas-Marini (BDM) mixed finite element discretization of
\[
  H_N(\bdiv,\Omega,\bbM) := \set*{\btau\in H(\bdiv, \Omega, \bbM); \quad 
	\dual*{\btau\bn,\boldsymbol{\phi}}_{\Gamma}= 0 
	\quad \text{$\forall\boldsymbol{\phi} \in H^{1/2}(\Gamma,\R^d)$,\, $\boldsymbol{\phi}|_{\Gamma_D} = \mathbf{0}$}}.
\] 
Given $s>1/2$ the tensorial version of the canonical BDM finite element interpolant  $\Pi^{\texttt{BDM}}_h: H_N(\sdiv,\Omega,\bbM)\cap H^s(\cup_j\Omega, \bbM)\to H_N(\sdiv,\Omega,\bbM)\cap \cP_k(\cT_h,\bbM)$ satisfies the following classical error estimate, \cite[Proposition 2.5.4]{BoffiBrezziFortinBook},
\begin{equation}\label{asymp}
 \norm*{\btau - \Pi^{\texttt{BDM}}_h \btau}_{0,\Omega} \leq C h^{\min\{s, k+1\}} \sum_{j=1}^J \norm*{\btau}_{s,\Omega_j} \qquad \forall \btau \in H_N(\sdiv,\Omega,\bbM)\cap H^s(\cup_j\Omega, \bbM), \quad s>1/2,
\end{equation}
Moreover, we have the well-known commutativity property, 
\begin{equation}\label{commuting}
	\bdiv \Pi^{\texttt{BDM}}_h \btau = Q^{k-1}_h \bdiv \btau,\quad \forall \btau \in H_N(\sdiv,\Omega,\bbM)\cap H^s(\cup_j\Omega, \bbM), \quad s>1/2, 
\end{equation}
where $Q^{k-1}_h$ stands for the $L^2(\Omega, \R^d)$-orthogonal projection onto $\cP_{k-1}(\cT_h, \R^d)$. Therefore, if $\bdiv \btau \in H^s(\cup_j\Omega,\R^d)$, we obtain
\begin{equation}\label{asympDiv}
 \norm{\bdiv (\btau - \Pi^{\texttt{BDM}}_h \btau) }_{0,\Omega} = \norm{\bdiv \btau - Q^{k-1}_h \bdiv \btau }_{0,\Omega} 
 \leq C h^{\min\{s, k\}} \sum_{j=1}^J \norm{\bdiv\btau}_{s,\Omega_j}.
\end{equation}

We point out that one can actually extend  the domain of the canonical interpolation operator $\Pi^{\texttt{BDM}}_h$ to  $H_N(\bdiv,\Omega,\bbM)\cap H^s(\Omega,\bbM)$, for any $s>0$. In the case of a constant function $\varrho$ and a constant tensor $\cA$, classical regularity results \cite{dauge, grisvard} ensure the existence of $\hat s\in(0,1]$ (depending on $\Omega$ on the boundary conditions and on the Lam\'e coefficients) such that the solution $\widetilde{\bu}$ of problem \eqref{korn} belongs to $H^{1+s}(\Omega, \R^d)\cap H^1_D(\Omega, \R^d)$ for all $s\in (0, \hat s)$.   However, our aim here is to avoid relying on regularity results that may be difficult to establish for the elasticity system in the case of general domains, boundary conditions and material properties. For this reason, we resort to the following smoothed projector recently introduced by Licht \cite[Theorem 6.3]{licht}. 
\begin{theorem}\label{ern}
	There exists a bounded and linear operator $\mathcal J_h:\, L^2( \Omega, \bbM) \to H_N(\bdiv, \Omega, \bbM)\cap \cP_k(\cT_h, \bbM)$ such that 
	\begin{enumerate}[label=\roman*)]
		\item The exists $C>0$ independent of $h$ such that 
		\[
  \norm*{\bsig - \mathcal J_h \bsig}_{0,\Omega}\leq C \displaystyle\inf_{\btau_h\in H_N(\bdiv, \Omega, \bbM)\cap \cP_k(\cT_h, \bbM)} \norm*{\bsig - \btau_h}_{0,\Omega},\quad  \forall \bsig\in L^2( \Omega, \bbM)
  \]
\item $\bdiv \mathcal J_h \bsig = Q_h^{k-1} \bdiv \bsig$ for all $\bsig\in H_N(\bdiv, \Omega, \bbM)$.  
	\end{enumerate}
\end{theorem}

The operator $\mathcal{J}_h$ doesn't preserve symmetric. To remedy this drawback, we follow \cite{falk, guzman, wang} and use a symmetrisation procedure that requires the stability the Scott-Vogelius element \cite{vogelius} for the Stokes problem. We refer to \cite[Section 55.3]{ern1} for a detailed account on the conditions (on the mesh $\cT_h$ and $k$) under which this stability property is guaranteed in 2D and 3D. The analysis that follows from now on is based on the following assumption. 

\begin{assumption}\label{hyp}
The pair $\set*{\cP_{k+1}(\cT_h,\R^d)\cap H_N^1(\Omega,\R^d),  \cP_{k}(\cT_h)}$ is stable for the Stokes problem on the mesh $\cT_h$: there exists $\beta>0$ independent of $h$ such that 
\begin{equation}\label{infsupS}
\sup_{\bv_h\in \cP_{k+1}(\cT_h,\R^d)\cap H_N^1(\Omega,\R^d)} \frac{(\sdiv \bv_h, \phi_h )}{\norm{\bv_h}_{1,\Omega}} \geq \beta \norm{\phi_h}_{0,\Omega}\quad \forall \phi_h\in \cP_{k}(\cT_h).	
\end{equation}
	  
\end{assumption}

\begin{lemma}\label{stokes}
Under Assumption~\ref{hyp}, there exists a linear operator
\[
  \mathcal S_h:\, \cP_k(\cT_h,\bbM)\cap H_N(\bdiv,\Omega,\bbM)\to \cP_{k}(\cT_h,\bbS)\cap X
  \]
such that, for all $\btau_h \in \cP_k(\cT_h,\bbM)\cap H_N(\bdiv,\Omega,\bbM)$,
	\begin{enumerate}[label=\roman*)]
		\item $\bdiv ( \btau_h - \mathcal S_h \btau_h) = \mathbf 0$ in $\Omega$, 
		\item  and 
		$
  \norm*{\btau_h - \mathcal S_h \btau_h}_{0,\Omega}\leq C \norm*{\btau_h - \btau_h^{\mt}}_{0,\Omega}, \quad \text{with $C>0$ independent of $h$}.
  $  
	\end{enumerate}
\end{lemma}
\begin{proof}
We only sketch the proof given in \cite[Lemma 5.2]{wang} and adapt it to our boundary conditions, see also \cite{falk, guzman}. In the case $d=2$, given $\btau_h\in \cP_k(\cT_h,\bbM)\cap H_N(\bdiv,\Omega,\bbM)$, it follows from Assumption~\ref{hyp} that there exists  $\bw_h\in H^1_N(\Omega, \R^d)\cap \cP_{k+1}(\cT_h,\R^d)$ satisfying $\sdiv \bw_h = \tau_{h,21} - \tau_{h,12}$ and 
\begin{equation}\label{lom}
	\norm{\bw_h}_{1,\Omega} \lesssim \norm{\tau_{h,12} - \tau_{h,21}}_{0,\Omega} \lesssim \norm{\btau_{h} - \btau_{h}^{\mt}}_{0,\Omega}.
\end{equation}
We recall that all differential operators are applied row-wise and let  $\cS_h\btau_h := \btau_h + \nabla^\bot \, \bw_h$, where $\nabla^\bot:= (-\partial_2 , \partial_1 )^\mt$ is the rotated gradient. By construction,  $\bdiv (\btau_h - \cS_h\btau_h) = \mathbf 0$ and thanks to \eqref{lom} it holds $\norm{\btau_h - \cS_h\btau_h}_{0,\Omega} \lesssim \norm{\btau_{h} - \btau_{h}^{\mt}}_{0,\Omega}$. Moreover, it is easy to check that $\cS_h\btau_h = (\cS_h\btau_h)^\mt$. It remains to show that $\cS_h$ preserves the boundary condition on $\Gamma_N$. This follows from the fact that $\nabla^\bot\bw_h\bn  =  (\partial_\tau w_{h,1}, \partial_\tau w_{h,2})^\mt$ and the tangential derivatives $\partial_\tau w_{h,j}:= \partial_1 w_{h,j} n_2  - \partial_2 w_{h,j} n_1$, $j=1,2$, vanish on $\Gamma_N$ since $\bw_h\in H^1_N(\Omega,\R^d)$. This finishes the proof of the result in the two dimensional case. 

In the case $d=3$, we let $\cS_h\btau_h := \btau_h + \nabla \times \bw_h$, with $\bw_h = \mathbf{z}_h^\mt - (\tr{\mathbf{z_h}})I$, where the tensor $\mathbf{z}_h\in H^1_N(\Omega, \bbM)\cap \cP_{k+1}(\cT_h,\bbM)$  satisfies $\bdiv \mathbf{z}_h = (\btau_{h,23} - \btau_{h,32}, \btau_{h,31} - \btau_{h,13}, \btau_{h,12} - \btau_{h,21} )^\mt$ and $\norm{\mathbf{z}_h}_{1,\Omega}  \lesssim \norm{\btau_{h} - \btau_{h}^{\mt}}_{0,\Omega}$. The existence of $\mathbf{z}_h$  is ensured by Assumption~\ref{hyp}.  In this way, we also have $\bdiv (\btau_h - \cS_h\btau_h) = \mathbf 0$ and $\norm{\btau_h - \cS_h\btau_h}_{0,\Omega} \lesssim \norm{\btau_{h} - \btau_{h}^{\mt}}_{0,\Omega}$. The proof of the symmetry property $\cS_h\btau_h = (\cS_h\btau_h)^\mt$ is a little more involved in this case, as shown in \cite[Lemma 5.2]{wang}. Finally, we point out that $(\nabla \times \bw_h)\bn  = (\sdiv_\Gamma (\bw_h^1\times \bn), \sdiv_\Gamma (\bw_h^2\times \bn), \sdiv_\Gamma (\bw_h^3\times \bn))^\mt $ on $\Gamma$, where $\bw_h^j$, $j=1,2,3$ stand for the rows of $\bw_h$ and $\sdiv_\Gamma$ represents the divergence operator on the surface $\Gamma$. Taking into account that, $\bw_h |_{\Gamma_N} = \mathbf 0$, we deduce that $\cS_h\btau_h$ belongs to $\cP_{k}(\cT_h,\bbS)\cap X$ for all $\btau_h\in \cP_k(\cT_h,\bbM)\cap H_N(\bdiv,\Omega,\bbM)$, which finishes the proof of the result in the tree-dimensional case.  
\end{proof}

We are able to state now the counterpart of Theorem~\ref{ern} for $\cJ^s_h:= \cS_h\circ \cJ_h:\, L^2(\Omega, \bbM)\to X_h^c$.  

\begin{corollary}\label{inter}
Under Assumption~\ref{hyp},  $\mathcal J^s_h:\, L^2( \Omega, \bbS) \to  X_h^c$ satisfies 
	\begin{enumerate}[label=\roman*)]
		\item $
  \norm*{\bsig - \mathcal J^s_h \bsig}_{0,\Omega}\leq C \displaystyle\inf_{\btau_h\in H_N(\bdiv, \Omega, \bbM)\cap \cP_k(\cT_h, \bbM)} \norm*{\bsig - \btau_h}_{0,\Omega},\quad  \forall \bsig\in L^2( \Omega, \bbS)
  $
  
  with $C$ independent of $h$,
\item and $\bdiv \mathcal J^s_h \bsig = Q_h^{k-1} \bdiv \bsig$ for all $\bsig\in X$.  
	\end{enumerate}
\end{corollary}
\begin{proof}
The commuting property for $\cJ^s$ follows from the corresponding property for $\cJ_h$ and from  the fact that $\cS_h$ preserves the divergence of tensors, as stated in Lemma~\ref{stokes} i). 
  In addition, as a consequence the property given by Lemma~\ref{stokes} ii), for any $\bsig\in L^2(\Omega, \bbS)$, it holds  
\begin{align}\label{hu1}
\begin{split}
	\norm*{\bsig - \cJ^s_h \bsig }_{0,\Omega} &\leq \norm*{\bsig - \cJ_h \bsig }_{0,\Omega} + \norm*{\cJ_h \bsig - \mathcal S_h(\cJ_h \bsig )}_{0,\Omega}
	\\[1ex]
	& \lesssim \norm*{\bsig - \cJ_h \bsig }_{0,\Omega} + \norm*{\bsig - \cJ_h \bsig - (\bsig - \cJ_h \bsig)^\mt}_{0,\Omega}\lesssim \norm*{\bsig - \cJ_h \bsig }_{0,\Omega},
	\end{split}
\end{align}
and the first statement of the Corollary follows from Theorem~\ref{ern} i). 
\end{proof}

\begin{remark}\label{rem1}
Using the density of smooth functions in $H_N(\bdiv,\Omega, \bbM)$ \cite[Lemma 1.2]{licht}  and the interpolation error estimates satisfied by the BDM projector, we deduce from Corollary~\ref{inter} that, for all $\bsig\in X$, 
\[
  \inf_{\btau_h\in X_h^c} \norm*{\bsig - \btau_h}_X\leq  \norm*{\bsig - \cJ_h^s\bsig}_X \lesssim \inf_{\btau_h\in H_N(\bdiv, \Omega, \bbM)\cap \cP_k(\cT_h, \bbM) } \norm*{\bsig - \btau_h}_{0,\Omega} + \norm{\bdiv\bsig - Q_h^{k-1} \bdiv\bsig}_{0,\Omega} \to 0 
\]
when $h$ goes to zero,  which proves \eqref{approx}.  
\end{remark}

\subsection{The operator $P_h$}
In what follows, the norm of a linear and continuous operator $L:\, V_1 \to V_2$ between two Hilbert spaces $V_1$ and $V_2$ is denoted $\norm{L}_{\mathcal L(V_1, V_2)} := \sup_{v\in V_1, \norm{v}_{V_1}=1}\norm{Lv}_{V_2}$.

It is crucial to notice that (under Assumption~\ref{a0}) Proposition~\ref{SPD} provides the following equivalent characterization of $K_h$   
\[
K_h := \set*{\btau_h\in X_h^c;\ \bdiv \btau_h = 0 \ \text{in $\Omega$}} \subset K,
\] 
and that its $a(\cdot, \cdot)$-orthogonal complement $K_h^\bot := \set*{\bsig_h\in X_h^c;\ \inner*{ \bsig_h, \btau_h}_\cA = 0,\ \forall \btau_h \in K_h}$ is not a subset of $K^\bot= P(X)$. Let $P_h:X_h^c \to K_h^\bot$ be the $X$-orthogonal projection in $X_h^c$ onto $K_h^\bot$. The following result provides an estimate for the operator $(P- P_h)|_{X^c_h}$.

 
 \begin{lemma}\label{p_h}
 Under Assumptions \ref{a0} and \ref{hyp}, it holds
 	\[
  \norm*{P - P_h}_{\cL(X^c_h, X)} \leq 2 \norm*{(I - \cJ^s_h)P}_{\cL(X, H)}.
\]
 \end{lemma}
\begin{proof}
Let us first notice that, by definition of $P$ and $\cJ^s_h$,  for any $\btau_h\in X_h^c$, 
\[
  \bdiv ( \btau_h - \cJ^s_hP\btau_h) =  \bdiv \btau_h - \bdiv \cJ_h P \btau_h = \bdiv \btau_h - Q_h^{k-1}  \bdiv P \btau_h =  \bdiv \btau_h - Q_h^{k-1}  \bdiv  \btau_h = \mathbf 0, 
\]
 which proves that $(I - \cJ^s_hP)X_h^c\subset K_h$. Hence, it follows from the triangle inequality that 
\begin{align}\label{plug}
\begin{split}
	\norm*{(P - P_h)\btau_h}_X &\leq \norm*{ P_h\btau_h - \cJ^s_h P \btau_h }_X + \norm*{ (I - \cJ^s_h) P \btau_h }_X
	\\[1ex]
	& = \norm*{ P_h \btau_h- \cJ^s_h P \btau_h }_\cA + \norm*{ (I - \cJ^s_h) P \btau_h }_\cA, \quad \forall \btau_h\in X_h^c,
\end{split}
\end{align}
where we took into account that $P_h\btau_h - \cJ^s_h P \btau_h = \btau_h - \cJ^s_h P \btau_h - (\btau_h - P_h\btau_h) \in K_h$ and  
\[
  P \btau_h - \cJ^s_h P \btau_h = \btau_h - \cJ^s_h P \btau_h - (\btau_h - P\btau_h) \in K.
\]
To estimate the first term in the right-hand side of \eqref{plug}, we take advantage of the inclusion $K_h \subset K$ to write 
\begin{align*}
  \inner*{ P_h \btau_h- \cJ^s_h P \btau_h, P_h \btau_h- \cJ^s_h P \btau_h }_\cA =
  \inner*{ P \btau_h - \cJ^s_h P \btau_h, P_h \btau_h - \cJ^s_h P \btau_h }_\cA
\end{align*}
and we deduce from the Cauchy-Schwarz inequality that 
\[
  \norm*{ P_h \btau_h - \cJ^s_h P \btau_h }_\cA \leq \norm*{  (I - \cJ^s_h) P \btau_h }_\cA.
\]
Plugging the last estimate in \eqref{plug} gives the result. 
\end{proof}

\begin{remark}
A conforming approximation of problem \eqref{S} based on $X_h^c$ is not useful in practice since it is not straightforward to construct an explicit basis of this finite element space. 
\end{remark}

\section{Spectral correctness of the DG scheme and error estimates}\label{sec5}

\subsection{The main result}
Even in the case of conforming Galerkin approximations of eigenproblems, it is well-known \cite{ActaBoffi} that when the source operator is not compact, a convergent discrete scheme for the source problem doesn't necessarily provide a correct approximation of the spectrum. A fortiori, in our case, Theorem~\ref{cea} is not enough to prevent \eqref{Sh} from producing  spurious eigenvalues.  The procedure introduced in \cite{DNR1,DNR2} to analyse the spectral approximation of non compact operators has been recently adapted in  \cite[Section 5]{LMMR} to a DG context (cf. also \cite{BuffaPerugia}). It is shown that the main ingredient to prove the spectral correctness of the method  is the uniform convergence of $\tilde T_h$ to $\tilde T$ with respect to the following $h$--dependent norm,
\begin{equation}\label{main}
	\norm{\tilde T - \tilde T_h}_h:= \sup_{\btau_h\in \cP_k(\cT_h, \bbS)}\frac{\vertiii{(\tilde T - \tilde T_h)\btau_h}}{\vertiii{\btau}}\longrightarrow 0,\quad \text{when $h\to 0$}.
\end{equation}
 We need the following technical result to prove \eqref{main}. 

\begin{lemma}\label{propC}
Under Assumption~\ref{hyp}, there exists a projector $\mathcal{X}^s_h:\, X^{DG}_h \to X_h^c$ such that
\begin{equation}\label{equivN}
h \norm*{\bdiv_h(\btau_h - \cX_h^s \btau_h)}_{0,\Omega} +  \norm*{\btau_h - \cX_h^s \btau_h}_{0,\Omega} \leq C h 
 \norm*{h_{\cF}^{-1/2} \jump{\btau_h}}_{0,\cF^*_h}\quad \forall \btau_h \in X_h^{DG}, 
\end{equation}
 with $C>0$ independent of $h$.
\end{lemma}
\begin{proof}
It is proved in \cite[Proposition 5.2]{MMT} that there exists a projector $\mathcal{X}_h:\, X^{DG}_h \to H_N(\bdiv, \Omega, \bbM)\cap\cP_k(\cT_h,\bbM)$ such that 
\begin{equation}\label{equivN0}
  h \norm*{\bdiv_h(\btau_h - \cX_h \btau_h)}_{0,\Omega} +  \norm*{\btau_h - \cX_h \btau_h}_{0,\Omega} \leq C h 
 \norm*{h_{\cF}^{-1/2} \jump{\btau_h}}_{0,\cF^*_h}\quad \forall \btau_h \in X_h^{DG}, 
\end{equation}
By construction of $\cS_h$, the operator $\cX_h^s := \cS_h\circ \cX_h$ satisfies  $\bdiv \cX_h^s \btau_h = \bdiv \cX_h \btau_h$. Morover,  using property ii) of Lemma~\ref{stokes}  and reasoning as for estimate \eqref{hu1} yields 
\[
  \norm*{\btau_h - \cX_h^s \btau_h}_{0,\Omega} \lesssim  \norm*{\btau_h - \cX_h \btau_h}_{0,\Omega}. 
\]
It follows that \eqref{equivN} is a direct consequence of \eqref{equivN0}.
\end{proof}

We point out that the stability of $\mathcal{X}^s_h:\, X^{DG}_h \to X_h^c$ follows directly from the triangle inequality and \eqref{equivN}, namely,
\begin{equation}\label{stabX}
	\norm*{\cX_h^s \btau_h}_X \leq C \vertiii{\btau_h},\quad \forall \btau_h \in X_h^{DG},
\end{equation} 
with $C>0$ independent of $h$.

We are now in a position to prove the main result of this article.

\begin{theorem}\label{intermediate}
Under Assumptions \ref{a0} and \ref{hyp}, it holds 
	\[
\norm*{\tilde T - \tilde T_h}_h \leq C \left( h +  \norm*{(I - \cJ^s_h)P}_{\cL(X, H)} + \norm{(I - \cJ_h^s)TP}_{\cL(X, X)}\right),
\]
with $C$ independent of $h$. 
\end{theorem}
\begin{proof}
For any $\btau_h\in X_h^{DG}$, we consider the splitting $\btau_h = (I - \cX_h^s)\btau_h + P_h\cX_h^s\btau_h + (I - P_h) \cX_h^s\btau_h$ and exploit the fact that $(I - P_h) \cX_h^s\btau_h\in K_h\subset K$ is in the kernel of $(\tilde T - \tilde T_h)$  to obtain 
	\begin{align*}
	(\tilde T - \tilde T_h)\btau_h &= (\tilde T - \tilde T_h)(\cI - \cX_h^s)\btau_h + (\tilde T - \tilde T_h)P_h\cX_h^s\btau_h
	\\[1ex]
	&=(\tilde T - \tilde T_h)(I - \cX_h^s)\btau_h + (\tilde T - \tilde T_h)(P_h - P)\cX_h^s\btau_h + (\tilde T - \tilde T_h)P\cX_h^s\btau_h.
\end{align*}
It follows from the triangle inequality, \eqref{cotaT} and \eqref{bTh} that 
\begin{equation}\label{split3}
	\vertiii{(\tilde T - \tilde T_h)\btau_h} \leq 3\norm*{(\cI - \cX_h^s)\btau_h}_\cA + 3  \norm{ (P - P_h)\cX_h^s\btau_h}_{\cA} + \vertiii{ (\tilde T - \tilde T_h)P\cX_h^s\btau_h }.
\end{equation}
Using \eqref{equivN}, we can bound the first term in the right-hand side of \eqref{split3} as follows,
\begin{align}\label{or1}
	\norm*{(I - \cX_h^s)\btau_h}_\cA& \leq  a^+ \norm{ (I - \cX_h^s)\btau_h }_{0,\Omega}
	 \leq C h \vertiii{\btau_h}.
\end{align}
For the second term,  \eqref{stabX} and Lemma~\ref{p_h} yield 
\begin{align}\label{or2}
	\norm*{(P - P_h)\cX_h^s\btau_h}_\cA & \leq  \norm*{P - P_h}_{\cL(X^c_h, X)}  \norm*{\cX_h^s\btau_h}_X \lesssim \norm*{(I - \cJ^s_h)P}_{\cL(X, H)} \vertiii{\btau_h}.
\end{align}

To bound the third term in the right-hand side of \eqref{split3}, we begin by applying Céa estimate \eqref{Cea} to obtain  
\begin{equation}\label{porCea}
	\vertiii{ (\tilde T - \tilde T_h)P\cX_h^s\btau_h } \leq ( 1 + 2 M ) \vertiii{(I - \cJ_h^s) TP\cX_h^s\btau_h}_*.
\end{equation}
Let us introduce the notation $\bu:= \frac{1}{\varrho}\bdiv (TP\cX_h^s \btau_h) \in H_D^1(\Omega,\R^d)$ and notice that, by virtue of \eqref{regT} and \eqref{stabX},  
\begin{equation}\label{uu}
	\norm{\bu}_{1,\Omega} \leq C \norm*{ P\cX_h^s \btau_h }_\cA \leq C \norm*{\cX_h^s \btau_h }_X \leq C_1 \vertiii{\btau_h}. 
\end{equation}
Moreover, taking into account that $\varrho$ is piecewise constant and using Corollary~\ref{inter} ii) we can write  $\frac{1}{\varrho}\bdiv (I - \cJ_h^s) TP\cX_h^s\btau_h = \bu  - \frac{1}{\varrho}Q_h^{k-1} (\bdiv TP\cX_h^s\btau_h) = (I - Q_h^{k-1})\bu$ and it follows that  
	\begin{align}\label{root}
	\vertiii{(I - \cJ_h^s) TP\cX_h^s\btau_h}^2_*=  \norm*{(I - \cJ_h^s)TP\cX_h^s \btau_h }^2_X  +  \norm*{\sqrt{\varrho_\cF h_\cF} \mean{(I - Q_h^{k-1})\bu}}^2_{0,\cF_h^*}.
	\end{align}
Now, from the one hand, 
\begin{equation}\label{mo1}
	\norm*{(I - \cJ_h^s) TP\cX_h^s\btau_h}_X \leq \norm{(I - \cJ_h^s)TP}_{\cL(X, X)}\norm*{\cX_h^s \btau_h }_X \lesssim  \norm{(I - \cJ_h^s)TP}_{\cL(X, X)}\vertiii{ \btau_h }
\end{equation}\label{mo2}
and from the other hand, a classical scaling argument combined with \eqref{uu} yields
	\begin{align}\label{mo2}
	  \norm*{\sqrt{ h_\cF} (\bu - Q_h^{k-1}\bu)}_{0,\cF_h^*}
	\lesssim   h |\bu|_{1,\Omega} \leq C h \vertiii{ \btau_h }.
\end{align}
Using \eqref{mo1} and \eqref{mo2} in \eqref{porCea} gives the estimate 
\begin{equation}\label{or3}
	\vertiii{(I - \cJ_h^s) TP\cX_h^s\btau_h}_* \leq C \left( h + \norm{(I - \cJ_h^s)TP}_{\cL(X, X)} \right) \vertiii{ \btau_h }. 
\end{equation}
Finally, plugging \eqref{or1},  \eqref{or2}, and \eqref{or3} in \eqref{split3} gives the result.


\end{proof}

\begin{corollary}\label{T-Th}
	Under Assumptions \ref{a0} and \ref{hyp}, it holds 
	\[
\lim_{h\to 0} \norm*{\tilde T - \tilde T_h}_h=0.
\]
\end{corollary}
\begin{proof}
The pointwise convergence of $I - \cJ_h^s:\, X \to X$ to zero (ensured by Corollary~\ref{inter}) and  the compactness of $P:\, X \to H$  and $TP:\, X \to X$  imply that the operators $(I - \cJ_h^s)P:\, X\to H$ and $(I - \cJ_h^s)TP:\, X\to X$ are uniformly convergent to zero; namely,
\[
  \lim_{h\to 0} \norm*{(I - \cJ_h^s)P}_{\mathcal{L}(X, H)} = 0,\quad \text{and}\quad \lim_{h\to 0} \norm*{(I - \cJ_h^s)TP}_{\mathcal{L}(X,X)} = 0,
\]
and the result follows directly from Theorem~\ref{intermediate}.
\end{proof}

\subsection{Spectral correctness and convergence}
For the sake of completeness, in the remainder of this section we show (by applying a number of results from \cite[Section 5]{LMMR}) how to exploit property \eqref{main} to derive the correct spectral convergence of \eqref{Sh}. Let us first introduce some notations. For $\bsig\in X^{DG}(h)$ and $E$ and $F$ closed subspaces of $X^{DG}(h)$, we set $\delta(\bsig, E):=\inf_{\btau\in E}\vertiii{\bsig - \btau}$, $\delta( E, F):=\sup_{\bsig \in E:\,\vertiii{\bsig}=1}\delta(\bsig, F)$, and $\gap( E, F):=\max\set{\delta( E, F),\delta( F, E)}$, the latter being the so called \textit{gap} between subspaces $E$ and $F$. 

Let $\Lambda\subset \mathbb C\setminus \set{0,1}$ be an arbitrary compact set with smooth boundary $\partial \Lambda$ satisfying $\partial \Lambda \cap \sp(T) = \emptyset$. We assume that there are $m$ eigenvalues   $\eta^\Lambda_1, \ldots, \eta^\Lambda_m$ of $T$ (repeated according to their algebraic multiplicities)   inside $\partial \Lambda$. The following result shows that the resolvent $\left(z I- \tilde T\right)^{-1}: X^{DG}(h)\longrightarrow X^{DG}(h)$ is uniformly bounded with respect to $h$ and $z\in \partial \Lambda$.

\begin{lemma}\label{TDG}
There exists a constant $C>0$ independent of $h$ such that 
\[
\vertiii{(z I - \tilde T) \btau } \geq C\, \vertiii{\btau} \quad \forall \btau \in X^{DG}(h), 
\]
for all $z\in \partial \Lambda$.
\end{lemma}
\begin{proof}
	See \cite[Lemma 3.2]{LMMR}
\end{proof}
 We deduce from Lemma~\ref{TDG} that the operator $\cE:=\frac{1}{2\pi i}\displaystyle\int_{\partial\Lambda}\left(z I- \tilde T\right)^{-1}\, dz: X^{DG}(h)\longrightarrow X^{DG}(h)$ is well-defined and bounded uniformly in $h$. Moreover, $\cE|_{X}:\, X \to X$ is a projector onto the finite dimensional space $\cE(X)$ spanned by the generalized eigenfunctions associated with the finite set of eigenvalues of $T$ contained in $\Lambda$. Actually, it is easy to check that $\tilde T: X^{DG}(h) \to X^{DG}(h)$ and $T: X \to X$ have the same eigenvalues in $\Lambda$ and that $\cE(X^{DG}(h)) = \cE(X)$. 

The next step consists in combining Lemma~\ref{TDG} and Theorem~\ref{T-Th} to deduce that the discrete resolvent $\left(z I- \tilde T_h\right)^{-1}: X^{DG}(h)\longrightarrow X^{DG}(h)$ is also uniformly bounded, provided $h$ is small enough, cf.  \cite[Lemma 5.1]{LMMR} for more details.

\begin{lemma} \label{ThDG}
Under Assumptions \ref{a0} and \ref{hyp}, there exists $h_0>0$ such that for all $h\leq h_0$,
\[
\vertiii{(z I - \tilde T_h) \btau } \geq C\, \vertiii{\btau} \quad \text{for all $\btau \in X^{DG}(h)$ and $z\in \partial \Lambda$} 
\]
with $C>0$ independent of $h$.
\end{lemma}

Here again, it follows from Lemma~\ref{ThDG} that, for $h$ small enough, the linear operator  
\[
\cE_h:=\frac{1}{2\pi i}
\int_{\Lambda}\left(z I- \tilde T_h\right)^{-1}\, dz: X^{DG}(h)\longrightarrow X^{DG}(h)
\] 
is uniformly bounded in $h$. Likewise,  $\cE_h|_{X^{DG}_h}: X^{DG}_h\to X^{DG}_h$ is a projector onto the $\tilde T_h$-invariant subspace $\cE_h(X^{DG}(h)) = \cE_h(X^{DG}_h)$  corresponding to the eigenvalues of $T_h:\, X^{DG}_h \to X^{DG}_h$ contained in $\Lambda$.  

The approximation properties of the eigenfunctions of problem~\eqref{S} by means of those of problem~\eqref{Sh} are obtained as a consequence of the following estimate of the distance between $\cE_h(X^{DG}_h)$ and $\cE(X)$, measured in terms of the gap $\gap$. 

\begin{theorem}
Under Assumptions \ref{a0} and \ref{hyp}, there exists $h_0>0$ such that for all $h\leq h_0$, 
\begin{equation}\label{delta}
	\gap(\cE(X), \cE_h(X^{DG}_h)) \leq C \left( \norm{\tilde T - T_h}_h + \delta(\cE(X), X^{DG}_h)\right),
\end{equation}
with $C>0$ independent of $h$. 
\end{theorem}
\begin{proof}
	See \cite[Theorem 5.1]{LMMR}. 
\end{proof}

We point out that since $\cE(X)$ is a finite dimensional subspace of $X$, Remark~\ref{rem1} and Theorem~\ref{T-Th} ensure the convergence of $\gap(\cE(X), \cE_h(X^{DG}_h))$ to zero when $h\to 0$. This is the main ingredient in the proof of the following Theorem, cf. \cite[Theorem 5.2]{LMMR} for more details. 

\begin{theorem}\label{theolim}
Assume that Assumptions \ref{a0} and \ref{hyp} are satisfied. Let $\Lambda\subset \mathbb C\setminus \set{0,1}$ be an arbitrary compact set with smooth boundary $\partial \Lambda$ satisfying $\partial\Lambda \cap \sp(T) = \emptyset$. We assume that there are $m$ eigenvalues   $\eta^\Lambda_1, \ldots, \eta^\Lambda_m$ of $T$ (repeated according to their algebraic multiplicities)   contained in $\Lambda$. We also consider the eigenvalues $\eta^\Lambda_{1, h}, \ldots, \eta^\Lambda_{m(h), h}$ of $T_h:\, X^{DG}_h \to X^{DG}_h$  lying in  $\Lambda$ and repeated according to their algebraic multiplicities. Then, there exists $h_0>0$ such that $m(h) = m$ for all $h\leq h_0$  and 
\[
\lim_{h\to 0} \max_{1\leq i \leq m} |\eta_i^\Lambda -  \eta^\Lambda_{i, h}| =0.
\]
Moreover, if $\cE(X)$ is the $T$-invariant subspace of $X$ spanned by the generalized eigenfunctions corresponding to the set of eigenvalues $\set{\eta^\Lambda_{i},\ i = 1,\ldots, m}$ and $\cE_h(X^{DG}_h)$ is the $T_h$-invariant subspace of $X^{DG}_h$ spanned by the eigenspaces  corresponding to $\set{\eta^\Lambda_{i, h},\ i = 1,\ldots, m}$ then $\widehat\delta(\cE(X),\cE_h(X^{DG}_h))\to0$ as $h\to 0$.
\end{theorem}

\subsection{Error estimates for eigenvalues and eigenfunctions}

Theorem~\ref{theolim} guaranties that the discontinuous Galerkin scheme \eqref{Sh} does not pollute the spectrum of $T$ with  spurious modes. Moreover, it proves the  convergence of eigenvalues and eigenfunctions with correct multiplicity. However, in practice the space $\cE_\eta(X)$ of generalized eigenfunctions corresponding to a given isolated eigenvalue $\eta\neq 1$ enjoys individual smoothness properties and the term $\norm{\tilde T - T_h}_h$ in \eqref{delta} prevents from taking advantage of this specific regularity. For this reason, we are going to show now that the gap between the continuous and discrete eigenspaces corresponding to a particular eigenvalue $\eta\neq 1$ can be bounded only in terms of
\[
\delta^* (\cE(X) , X^{DG}_h):= \sup_{\bsig\in \cE(X), \norm{\bsig}_X=1}\inf_{\btau_h\in X^{DG}_h} \vertiii{\bsig - \btau_h}_*.
\]

Hence, hereafter we focus on a particular isolated eigenvalue $\eta\neq 1$ of $T$ of algebraic multiplicity $m$  and let $D_\eta\subset \mathbb C$ be a closed disk centered at $\eta$ with  boundary $\gamma$ such that $D_\eta \cap \sp(T) = \set{\eta}$. We denote by $\cE_\eta:=\frac{1}{2\pi i} \displaystyle\int_{\gamma}\left(z I- \tilde T\right)^{-1}\, dz: X \rightarrow X$ the projector onto the eigenspace $\cE_\eta(X)$ of $\eta$ and we define, for $h$ small enough, the projector by 
 $\cE_{\eta, h}:=\frac{1}{2\pi i} \displaystyle\int_{\gamma}\left(z I- \tilde T_h\right)^{-1}\, dz: X^{DG}_h \rightarrow X^{DG}_h$ onto the $T_h$-invariant subspace $\cE_{\eta, h}(X^{DG}_h)$  corresponding to the $m$ eigenvalues of $T_h:\, X^{DG}_h \to X^{DG}_h$ contained in $\gamma$. A straightforward adaptation of \cite[Theorem 6.1]{LMMR} gives the following result. 
\begin{theorem}
	Assume that Assumptions \ref{a0} and \ref{hyp} are satisfied. For $h$ small enough, there exists a constant $C$ independent of $h$ such that 
\begin{equation}\label{aste}
\gap\big(\cE_\eta(X), \cE_{\eta,h}(X^{DG}_h) \big)\leq C \delta^* (\cE_\eta(X) , X^{DG}_h).
\end{equation}
\end{theorem}

We conclude with the following rates of convergence for eigenfunctions and eigenvalues.
\begin{theorem}\label{errorE}
	Assume that Assumptions \ref{a0} and \ref{hyp} are satisfied. Let $r>0$ be such that $\cE_\eta(X) \subset \set*{ \btau\in H^r(\cup_j\Omega_j,\bbM);\ \bdiv\btau\in H^{1+r}(\cup_j\Omega_j, \R^d) }$. Then, there exists $C>0$ independent of $h$ such that, for $h$ small enough,
\begin{equation}\label{eigenspace}
\gap\big(\cE_\eta(X), \cE_{\eta,h}(X^{DG}_h) \big)\leq C h^{\min\{r,k\}}. 
\end{equation}
Moreover, there exists $C'>0$ independent of $h$ such that
\begin{equation}\label{eigenvalue}
\displaystyle\max_{1\leq i\leq m} |\kappa - \kappa_{i, h}|\leq  C' \, h^{2\min\{r,k\}},
\end{equation}
where $\kappa:= 1/\eta$ and $\kappa_{i,h}:= 1/\eta_{i,h}$, $i=1,\ldots, m$.
\end{theorem}
\begin{proof}
	For any $\bsig \in \cE_\eta(X)$, taking into account \eqref{commuting}, it holds
	\begin{align*}
		\inf_{\btau_h\in X^{DG}_h} &\vertiii{\bsig - \btau_h}_* \leq \vertiii{\bsig - \Pi^{\texttt{BDM}}_h\bsig}_* 
		\\[1ex]
		&= \left( \norm*{(I - \Pi^{\texttt{BDM}}_h)\bsig}^2_{\cA} + \norm*{\tfrac{1}{\sqrt \varrho}(I - Q_h^{k-1}) \bdiv\bsig}^2_{0,\Omega} + \norm*{\varrho_\cF^{\frac{1}{2}}h_\cF^{\frac{1}{2}} \mean{\tfrac{1}{\varrho}(I - Q_h^{k-1}) \bdiv\bsig}  }^2_{\cF_h^*}\right)^{1/2}.
	\end{align*}
Using \eqref{asymp}, \eqref{asympDiv} for the first and second terms of the last identity, respectively, and employing classical scaling arguments for the last one we deduce that 
\[
  \inf_{\btau_h\in X^{DG}_h} \vertiii{\bsig - \btau_h}_* \lesssim h^{\min\{r,k\} }\left(  \sum_{j=1}^J  \norm*{\bsig}^2_{r,\Omega_j} + \norm*{\bdiv \bsig}^2_{r,\Omega_j} \right)^{1/2},
\] 
and \eqref{eigenspace} follows from the fact that $\cE_\eta(X)$ is a finite dimensional subspace of $X$. 

The deduction of \eqref{eigenvalue} from \eqref{eigenspace} is obtained in the classical way, see  \cite[Theorem 6.2]{LMMR} for more details.
\end{proof}

\begin{remark}
	The displacement field corresponding to a given $\bsig \in \cE_\eta(X)$ is $\bu = - \frac{1}{\varrho (\kappa-1) }\bdiv \bsig$. Let us associate to $\bu$ the discrete displacement $\bu_h:= - \frac{1}{\varrho (\kappa_{h}-1) }\bdiv_h \bsig_h$, where $\kappa_{h}:= (\sum_{i=1}^m \kappa_{i,h})/m$. By virtue of the triangle inequality and Theorem~\ref{errorE}, for $h$ small enough, we have the error estimate  
	\begin{align*}
		\norm*{\bu - \bu_h}_{0,\Omega} \leq  \frac{1}{\varrho^-} \Big( \frac{1}{\kappa-1} \norm*{\bdiv_h(\bsig - \bsig_h)}_{0,\Omega} + \left| \frac{1}{\kappa - 1} - \frac{1}{\kappa_{h} - 1} \right| \norm*{\bdiv_h\bsig_h}_{0,\Omega} \Big)
		 \lesssim h^{\min\{r,k\}}.
	\end{align*}
 
\end{remark}

\section{Numerical results}\label{NR}\label{sec6}

We point out that the inf-sup condition \eqref{infsupS} corresponding to the Scott--Vogelius element is only known to be satisfied under certain conditions on the mesh and the polynomial degree $k$. Namely, in two dimensions, Assumption~\ref{hyp} holds true on shape-regular triangulations with no singular vertices for $k\geq 3$, cf. \cite{vogelius} and in dimendion three, it is satisfied  on uniform simplicial meshes for $k \geq 5$  \cite{zhang1}. For lower values $1 \leq k < 2d -1$ of the polynomial degree, we can ensure Assumption~\ref{hyp} by considering shape-regular meshes with barycentric refinements (Alfeld splits), see \cite {qin, zhang}. Similar results have been proved for meshes of Powell–Sabin type \cite{zhang2, zhang3}, which wil not be employed here.

 To our knowledge, the stability results mentioned so far for the Scott--Vogelius element have only be obtained for homogeneous Dirichlet or Neumann boundary conditions on $\Gamma$.  However, there is numerical evidence that the stability and optimal accuracy of this finite element method also occur on barycentric refinements of shape-regular meshes when mixed boundary conditions are imposed \cite{olshanskii, scott}.  
 
 In what follows, we say that $\cT^{\mathtt{bary}}_h$ is a simplicial barycentric ($d + 1$)-sected mesh of size $h$ if $\cT^{\mathtt{bary}}_h$ is obtained after refinement of a shape-regular simplicial mesh $\cT_h$ of size $h$ by subdividing each simplex in $\cT_h$ into $d+1$ sub-simplices by connecting the barycenter with the $d + 1$ vertices, see Figure~\ref{mesh}. 
 
 The numerical results presented in this section have been implemented using the finite element library \texttt{Netgen/NGSolve} \cite{ngsolve}.

\begin{figure}[h]
\begin{center}
\begin{minipage}{6cm}
\centerline{\includegraphics[height=6cm, angle=0]{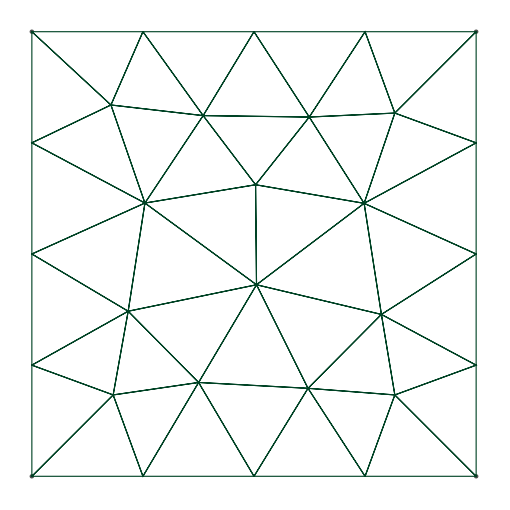}}
\end{minipage}
\begin{minipage}{6cm}
\centerline{\includegraphics[height=6cm, angle=0]{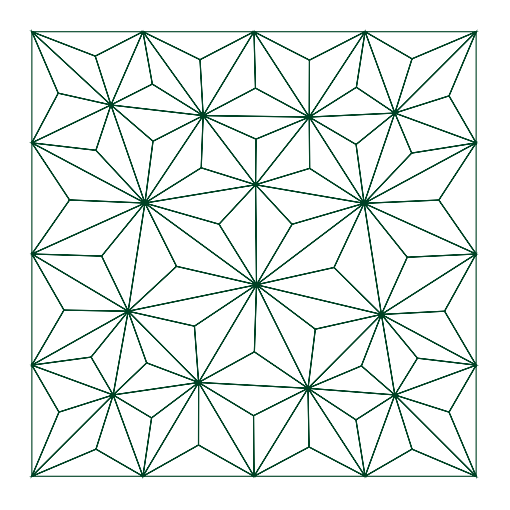}}
\end{minipage}
\caption{An unstructured shape-regular mesh $\cT_h$ of mesh size $h=1/4$ (left) and the corresponding barycentric refinement $\cT_h^{\mathtt{bary}}$ (right).}
\label{mesh}
\end{center}
\end{figure}

\medskip
\noindent\textbf{Example 1: Spectral correctness of the DG scheme in two dimensions.}  We assume that problem \eqref{modelPb} is posed in the unit square $\Omega = (0,1)^2$ and let the  compliance  tensor $\cA$ be given by Hooke's law
\begin{equation}\label{hook} 
\cA \btau = \frac{1}{2\mu}\btau - \frac{\lambda}{2\mu(d\lambda + 2\mu  )} \tr{\btau} I,\end{equation}
where $\lambda$ and $\mu$  are the Lam\'e coefficients. We select in this example constant values for the mass density $\varrho = 1$, Young's modulus $E=1$ and Poisson's ratio $\nu=0.35$. We recall that the Lam\'e coefficients are related to $E$ and $\nu$ by
\begin{equation*}
\label{LAME}
\lambda:=\frac{E\nu}{(1+\nu)(1-2\nu)}
\qquad\text{and}\qquad 
\mu:=\frac{E}{2(1+\nu)}.
\end{equation*}   
We assume that the solid is fixed at the bottom side $\Gamma_D = (0,1)\times \{0\}$ of the square  and free of stress on the remaining  three sides $\Gamma_N= \Gamma \setminus \Gamma_D$. 

\begin{table}[!htb]
\centering
\begin{minipage}{.45\textwidth}
\centering
\begin{tabular}{  c c c  }
\toprule
\multicolumn{3}{c}{$k=1$}\\

 $\mathtt{a}_0=4$  &  $\mathtt{a}_0=8$ &   $\mathtt{a}_0=16$   \\ 
\midrule
0.678702		&0.679772		& 0.680201\\
1.695659		&1.697598		& 1.698374\\
1.816904		&1.819964		& 1.821196\\
\textbf{2.639076}		&2.940891		& 2.944729\\
\textbf{2.837896}		&3.014975		& 3.016507\\
2.931343		&\textbf{3.401037}		& 3.442011\\
3.011368		&3.440911		& 4.138696\\
\textbf{3.108155}		&\textbf{3.651704}		& \textbf{4.157343}\\
3.417512		&\textbf{4.007783}		& \textbf{4.453139}\\
\textbf{3.441640}		&4.134962		& 4.628728\\

\bottomrule            
\end{tabular}
        \caption{Lowest natural frequencies on an unstructured  shape-regular triangulation $\cT_h$ of mesh size $h=1/64$.}
        \label{table1a}
    \end{minipage}\qquad %
    \begin{minipage}{0.45\textwidth}
        \centering
        \begin{tabular}{c c c  }
\toprule
\multicolumn{3}{c}{$k=1$}\\
  $\mathtt{a}_0=2$  &  $\mathtt{a}_0=4$ &   $\mathtt{a}_0=8$   \\ 
 \midrule
0.677490	&0.677848	&0.679280 \\
1.693684	&1.694236	&1.696863 \\
1.811491	&1.814970	&1.819734 \\
2.910592	&2.925604	&2.940824 \\
3.002259	&3.008271	&3.014647 \\
3.425010	&3.432744	&3.440763 \\
4.091722	&4.118359	&4.137677 \\
4.587870	&4.609321	&4.626700 \\
4.745053	&4.769440	&4.786882 \\
4.717068	&4.740651	&4.758608 \\
\bottomrule               
\end{tabular}
        \caption{Lowest natural frequencies on a barycentric refinement of  a shape-regular triangulation $\cT^{\mathtt{bary}}_h$ of mesh size $h=1/16$.}
        \label{table1b}
    \end{minipage}
\end{table}

We report in Table~\ref{table1a} the 10 smallest vibration frequencies $\omega_{hi}:=\sqrt{\kappa_{hi}-1}$ obtained by solving the DG scheme \eqref{Sh} on an unstructured shape-regular triangulation $\cT_h$ of mesh size $h=1/64$ at the lowest order $k=1$ and for stability parameters $\mathtt{a}=\mathtt{a}_0 k^2$, with $\texttt{a}_0:=4,8,16$. We observe spurious eigenvalues (the numbers in bold font) emerge at random positions, which indicates that the approximation is not spectrally correct. To ensure Assumption~\ref{hyp}, we solve now \eqref{Sh}   on a barycentric trisected mesh $\cT_h^{\mathtt{bary}}$ of size $h=1/16$. The results displayed in Table~\ref{table1b} (for $k=1$ and for stabilization parameters $\mathtt{a}=\mathtt{a}_0 k^2$, $\texttt{a}_0:=2,4,8$)  provide correct eigenfrequencies.

\begin{table}[!htb]
\centering
\begin{minipage}{.45\textwidth}
 \centering
\begin{tabular}{  c c c  }
\toprule
\multicolumn{3}{c}{$k=2$}\\
 $\mathtt{a}_0=2$  &  $\mathtt{a}_0=4$ &   $\mathtt{a}_0=8$   \\ 
\midrule
0.676422	&0.678614	&0.679493 \\
1.691646	&1.695661	&1.697263 \\
1.811504	&1.817566	&1.820003 \\
2.915251	&2.933614	&2.941019 \\
3.004017	&3.011501	&3.014480 \\
3.428332	&3.436800	&3.440173 \\
4.111159	&4.128585	&4.135536 \\
\textbf{4.561698}	&4.618693	&4.625052 \\
4.603006	&4.749476	&4.755874 \\
4.733581	&4.777890	&4.783112 \\
\bottomrule            
\end{tabular}
        \caption{Lowest natural frequencies on an unstructured  shape-regular triangulation $\cT_h$ of mesh size $h=1/16$.}
        \label{table2a}
    \end{minipage}\qquad %
    \begin{minipage}{0.45\textwidth}
        \centering
\begin{tabular}{c c c  }
\toprule
\multicolumn{3}{c}{$k=2$}\\
  $\mathtt{a}_0=2$  &  $\mathtt{a}_0=4$ &   $\mathtt{a}_0=8$   \\ 
 \midrule
0.679380	&0.680102	&0.680389 \\
1.696791	&1.698115	&1.698640 \\
1.818694	&1.820700	&1.821497 \\
2.937186	&2.943178	&2.945560 \\
3.013475	&3.015922	&3.016894 \\
3.438312	&3.441135	&3.442255 \\
4.131920	&4.137557	&4.139793 \\
4.621689	&4.627024	&4.629140 \\
4.752329	&4.757616	&4.759720 \\
4.780796	&4.785130	&4.786851 \\
\bottomrule               
\end{tabular}
        \caption{Lowest natural frequencies on a barycentric refinement of  a shape-regular triangulation $\cT^{\mathtt{bary}}_h$ of mesh size $h=1/16$.}
        \label{table2b}
    \end{minipage}
\end{table}

We repeat the same experiment by solving \eqref{Sh} with quadratic polynomial order. We employ an unstructured shape-regular  triangulation $\cT_h$ of size $h=1/16$ in Table~\ref{table2a} and a barycentric trisected mesh $\cT_h^{\mathtt{bary}}$ of size $h=1/16$ in Table~\ref{table2b}. Only one spurious eigenfrequency shows up in Table~\ref{table2a} among the first 10 eigenvalues in the case $\texttt{a}_0=2$. The DG method seems to provide a spectrally correct quadratic approximation on $\cT_h$ for $\mathtt{a}_0$ sufficiently large,  even though Assumption~\ref{hyp} is not known to be satisfied on $\cT_h$ for $k=2$.

\begin{table}[!htb]
\centering
\begin{minipage}{.45\textwidth}
 \centering
\begin{tabular}{  c c c  }
\toprule
\multicolumn{3}{c}{$k=3$}\\
 $\mathtt{a}_0=2$  &  $\mathtt{a}_0=4$ &   $\mathtt{a}_0=8$   \\ 
\midrule
0.676520	&0.678355	&0.679133 \\
1.692090	&1.695468	&1.696897 \\
1.812994	&1.818054	&1.820203 \\
2.919900	&2.935216	&2.941725 \\
3.004835	&3.011233	&3.013935 \\
3.430379	&3.437405	&3.440382 \\
4.115109	&4.129772	&4.135970 \\
4.605120	&4.619032	&4.624917 \\
\bottomrule            
\end{tabular}
        \caption{Lowest natural frequencies on an unstructured  shape-regular triangulation $\cT_h$ of mesh size $h=1/16$.}
        \label{table3}
    \end{minipage}\qquad %
    \begin{minipage}{0.45\textwidth}
        \centering
\begin{tabular}{c c c  }
\toprule
\multicolumn{3}{c}{$k=4$}\\
  $\mathtt{a}_0=2$  &  $\mathtt{a}_0=4$ &   $\mathtt{a}_0=8$   \\ 
 \midrule
0.678431	&0.679425	&0.679859 \\
1.695330	&1.697156	&1.697950 \\
1.817152	&1.819895	&1.821092 \\
2.932397	&2.940705	&2.944336 \\
3.010780	&3.014227	&3.015725 \\
3.436205	&3.440011	&3.441666 \\
4.127128	&4.135079	&4.138538 \\
4.616777	&4.624336	&4.627618 \\
\bottomrule               
\end{tabular}
        \caption{Lowest natural frequencies on an unstructured  shape-regular triangulation $\cT_h$ of mesh size $h=1/16$.}
        \label{table4}
    \end{minipage}
\end{table}

We finish this series of tests by reporting in Table~\ref{table3} and Table~\ref{table4} the eigenfrequencies obtain by solving \eqref{Sh} for $k=3$ and $k=4$, respectively. For these values of the polynomial order $k$, Assumption~\ref{hyp} is satisfied on unstructured shape-regular meshes $\cT_h$. We take $h=1/16$ and let  $\mathtt{a}=\mathtt{a}_0 k^2$ for $\texttt{a}_0:=2,4,8$ in each case. As expected, all the computed eigenfrequencies are correct. 

\medskip
\noindent\textbf{Example 2: Spectral correctness of the DG scheme in three dimensions.} We consider a solid represented by the unit cube $\Omega = (0,1)^3$ and impose a Dirichlet boundary condition on the whole boundary $\Gamma_D = \Gamma$. We maintain the same expression \eqref{hook} for the compliance tensor $\cA$ and the same constant coefficients $\rho=1$, $E=1$ and $\nu =0.35$ used in the previous example. 

Only spurious eigenvalues appeared when solving problem \eqref{Sh} on unstructured shape-regular simplicial partitions $\cT_h$ for $k=2,3,4$, for a wide range of parameters $\mathtt{a}=\mathtt{a}_0 k^2$, and for the largest number of degrees of freedom allowed by our computational capacity in each case. To guarantee Assumption~\ref{hyp}, we solved \eqref{Sh} on a quadrisected barycentric mesh $\cT_h^{\mathtt{bary}}$ of size $h=1/4$ and reported the results in Tables \ref{table3Da} and \ref{table3Db}.  The results displayed in Table~\ref{table3Da} indicate the spectral correctness of the DG scheme of quadratic order for $\mathtt{a}$ large enough. Finally, we list in Table~\ref{table3Db} the first natural frequencies obtained by solving \eqref{Sh} at cubic and quartic order with  $\mathtt{a}= 8k^2$. 
 
\begin{table}[!htb]
\centering
\begin{minipage}{.45\textwidth}
 \centering
\begin{tabular}{  c c c  }
\toprule
\multicolumn{3}{c}{$k=2$}\\
 $\mathtt{a}_0=4$  &  $\mathtt{a}_0=8$ &   $\mathtt{a}_0=16$   \\ 
\midrule
4.459599	&4.462669	&4.463426 \\
4.459644	&4.462324	&4.463592 \\
4.459717	&4.462452	&4.463901 \\
\textbf{4.735649}	&4.783271	&4.787701 \\
\textbf{4.856537}	&4.784776	&4.789726 \\
4.774039	&4.785016	&4.790044 \\
4.772765	&5.820485	&5.831800 \\
4.772953	&5.834638	&5.829005 \\
\textbf{5.188265}	&6.027316	&6.030225 \\
\textbf{5.288405}	&6.015911	&6.034207 \\
\bottomrule            
\end{tabular}
        \caption{Lowest natural frequencies on a barycentric refinement of  a shape-regular triangulation of mesh size $h=1/4$.}
        \label{table3Da}
    \end{minipage}\qquad %
    \begin{minipage}{0.45\textwidth}
        \centering
\begin{tabular}{c c   }
\toprule
\multicolumn{2}{c}{$\mathtt{a}_0=8$}\\
  $k=3$  &  $k=4$    \\ 
 \midrule
4.460305	&4.460220	 \\
4.460295	&4.460222	 \\
4.460286	&4.460221	 \\
4.770938	&4.770735	 \\
4.770963	&4.770734	 \\
4.770972	&4.770732	 \\
5.805414	&5.804214	 \\
5.881658	&5.804351	 \\
6.014326	&6.013368	 \\
6.014811	&6.017531	 \\
\bottomrule               
\end{tabular}
        \caption{Lowest natural frequencies on a barycentric refinement of  a shape-regular triangulation of mesh size $h=1/4$.}
        \label{table3Db}
    \end{minipage}
\end{table}

\medskip
\noindent\textbf{Example 3: Accuracy verification and stability in the nearly incompressible limit.} We have seen through Example 1 and Example 2 that, if Assumption~\ref{a0} and  Assumption~\ref{hyp}  are met, the DG scheme \eqref{Sh} does not pollute the spectrum of $T$ with spurious modes. The aim now is to confirm that  the   eigenvalues converge at the expected rate and with correct multiplicity.

We take $\varrho=1$, $E=1$ and $\cA$ given by \eqref{hook} and we let $\Gamma_D = \Gamma$. We observe that in the limit $\lambda \to \infty$ (or $\nu \to 0.5$), the eigenvalues of \eqref{S} converge to the eigenvalues of the following perfectly incompressible elasticity eigenproblem  (see \cite[Appendix 9]{LMMR}): find eigenmodes $0 \neq \bsig^\infty:\Omega\to \bbS$ and eigenvalues  $\zeta\in \mathbb{R}$ such that,
\begin{align}\label{modelPbS}
\begin{split}
  -\beps\left( \bdiv  \bsig^\infty \right) &= \frac{3\zeta}{2}  (\bsig^\infty)^{\tD}     \quad \text{ in $\Omega$},
 \\
\bdiv \bsig^\infty &= 0  \quad \text{ on $\Gamma$},\\
\inner*{\tr{\bsig^\infty}, 1} &=0,
\end{split}
\end{align}
where $\btau^{\tD}:=\btau-\frac{1}{d}\left(\tr\btau\right) I$ is the deviatoric part of a tensor $\btau$. Actually, \eqref{modelPbS} is the stress formulation of the  Stokes eigenproblem \cite[Section 6.2]{meddahi} with formal velocity and pressure fields given by $\bu^\infty = -\frac{2\mu}{\zeta}\bdiv \bsig^\infty$ and $p^\infty := -\frac{1}{d} \tr \bsig^\infty$, respectively.

\begin{table}[!htb]
\centering
\footnotesize
\begin{tabular}{l l c c c c c} 

\toprule  
 \textbf{$k$} & \textbf{$h$} & $3\zeta_{1h}$ & $3\zeta_{2h}$ & $3\zeta_{3h}$ & $3\zeta_{4h}$ & $3\zeta_{5h}$
 \\ 
 \midrule

\multirow{4}{*}{2} 
& 1/2 & 14.655874562841 & 26.871985748572  & 26.887435102120 & 27.832597710403 & 42.832773681162
\\ 
& 1/4 & 14.683077038284 & 26.407208192082  & 26.417044912441 & 40.851768588874 & 40.872250820064
\\
& 1/8 & 14.682182085061 & 26.376255232771  & 26.376328260081 & 40.712722345786 & 40.713400774810
\\
& 1/16 & 14.681986181069 & 26.374698550739 & 26.374704623714 & 40.706735329788 & 40.706760417670
\\ \bottomrule
$\texttt{avg}(\boldsymbol{r}_{h,2}^i)$  &  & \textbf{3.57} & \textbf{4.18} & \textbf{4.17} & \textbf{5.18}  & \textbf{4.27}
 \\ 
 \toprule
\multirow{4}{*}{3} 
& 1/2  & 14.685386398526 & 26.399150575171 & 26.399686773148 & 40.827777930236 & 40.828051578839
\\ 
& 1/4  & 14.682071040481 & 26.375381383674 & 26.375576473616 & 40.710758108868 & 40.711302839983
\\
& 1/8  & 14.681971611106 & 26.374624687636 & 26.374624985835 & 40.706510271654 & 40.706514798709
\\
& 1/16 & 14.681970657338 & 26.374616530921 & 26.374616534235 & 40.706466349720 & 40.706466354220
\\ \bottomrule
$\texttt{avg}(\boldsymbol{r}_{h,3}^i)$  &  & \textbf{5.92} & \textbf{5.94} & \textbf{5.94} & \textbf{5.93}  & \textbf{5.93}
 \\ 
 \toprule
\multirow{4}{*}{4} 
& 1/2  & 14.682036415561 & 26.375811477238 & 26.375817494769 & 40.715004945300 & 40.715041953610
\\ 
& 1/4  & 14.681971020753 & 26.374623059451 & 26.374625770791 & 40.706530102238 & 40.706538022701
\\
& 1/8  & 14.681970642864 & 26.374616440514 & 26.374616441501 & 40.706465931297 & 40.706465952803
\\
& 1/16 & 14.681970642114 & 26.374616427187 & 26.374616427190 & 40.706465818502 & 40.706465821494
\\ \bottomrule
$\texttt{avg}(\boldsymbol{r}_{h,4}^i)$ &  & \textbf{7.55} & \textbf{8.52} & \textbf{8.47} & \textbf{8.25}  & \textbf{7.10}
\\ 	\bottomrule
\end{tabular}\caption{Computed lowest eigenvalues $\zeta_{jh}$, $j=1,\ldots ,5$, of problem~\ref{modelPbS} and averaged rates of convergence for a set of unstructured shape-regular curved meshes $\widetilde\cT_h$ with decreasing mesh sizes $h$ and for polynomial degrees of approximation $k = 2,3,4$. The exact eigenvalues are given by \eqref{exactE}.}\label{tableS}
\end{table}

It turns out that, on the unit disk, the eigenvalues of the Stokes eigenproblem~\eqref{modelPbS} are given by the sequence  $\set*{\tfrac{1}{2}\jmath_{n\ell}^2}_{n\geq 1,\, \ell \geq 1}$, where $\jmath_{nk}$ is the $\ell$-$th$ positive zero of the Bessel function $J_n$ of the first kind of order $n$.   Accurate approximations of the first 4 eigenvalues are given by  
  \begin{align}\label{exactE}
  \begin{split}
  	3\zeta_1 &= \jmath_{11}^2 \simeq 14.681970642124\\ 3\zeta_2 = 3\zeta_3 &= \jmath_{21}^2 \simeq 26.374616427163 \\ 3\zeta_4 = 3\zeta_5 &= \jmath_{31}^2 \simeq 40.706465818200.
  \end{split}
  \end{align}

To deal with the completely incompressible case $\nu=0.5$, one can adapt the DG method \eqref{Sh} for problem \eqref{modelPbS} by changing the bilinear form $\inner*{\bsig, \btau}_\cA$ in \eqref{S}  to $\frac{1}{2\mu}\inner*{\bsig^\tD, \btau^\tD}$ and incorporating the constraint $\inner*{\tr{\bsig^\infty}, 1} =0$ into $X$, see \cite[Section 6.2]{meddahi} for more details. Here, with the aim to test the performance of the scheme in the nearly incompressible case,  we instead approximate the eigenvalues of \eqref{modelPbS} by solving the original DG method \eqref{Sh} with a Poisson's ratio  $\nu = 0.5 - 10^{-13}$.

We denote by $\zeta_{ih} = \sqrt{3(\kappa_{ih} - 1)}$ the approximation of $\zeta_i$ computed  by solving problem~\ref{Sh} on a series of exact meshes $\widetilde\cT_h$ of $\bar \Omega$  with decreasing mesh sizes $h$, and for polynomial degrees $k = 2, 3, 4$. The assembling of the  generalized eigenproblems corresponding to \eqref{Sh} is performed thanks to the support of Netgen/NGSolve \cite{ngsolve}  for curved finite elements  of arbitrary order. We present in Table~\ref{tableS} the first four computed eigenvalues  and report the arithmetic mean $\texttt{avg}(\boldsymbol{r}_{h,k}^i)$ of the three experimental rates of convergence, which  are obtained for each eigenvalue by mean of the formula 
\begin{equation}\label{rate}
	\boldsymbol{r}_{h,k}^i:= \dfrac{\log\inner*{\abs{\zeta_i - \omega^2_{ih}} / \abs{\zeta_j - \omega^2_{i\hat{h}}}  }}{\log(h/\hat h)},\quad i=1,\ldots, 4,\quad k=2,3,4, 
\end{equation}
where $h$ and $\hat h$ are two consecutive mesh sizes.

We observe that a convergence of order $2k$ is attained for each eigenvalue, as predicted by the error estimate \eqref{eigenvalue}. At the same time, this test shows that the DG-scheme \eqref{Sh} is inmune to locking in the nearly incompressible limit. 

\begin{remark}
	In principle, for $k=2$, we need to perform a barycentric refinement of the mesh $\widetilde\cT_h$ to ensure Assumption~\ref{hyp}. However, the optimal order of convergence reported in Table~\ref{tableS} in the quadratic case seems to confirm the results obtained in Table~\ref{table2a}. In other words,  our  numerical tests  suggest that, in the two-dimensional case, the DG scheme \eqref{Sh} of quadratic order provides a spectrally correct  approximation of \eqref{S} with optimal rates of convergence for the eigenvalues on shape-regular triangulations. This statement is not supported by our theory since \eqref{infsupS} is not known to be satisfied for $k=2$. 
\end{remark}

\end{document}